\documentclass{article}%
\usepackage{amsfonts}
\usepackage{amsmath}
\usepackage{amssymb}
\usepackage{graphicx}%
\newtheorem{theorem}{Theorem}

\newtheorem{corollary}[theorem]{Corollary}

\newtheorem{definition}[theorem]{Definition}
\newtheorem{example}[theorem]{Examples}

\newtheorem{lemma}[theorem]{Lemma}

\newtheorem{proposition}[theorem]{Proposition}
\newtheorem{remark}[theorem]{Remark}

\newenvironment{proof}[1][Proof]{\noindent\textbf{#1.} }{\ \rule{0.5em}{0.5em}}
\begin{document}

\title{Local real analysis \\in locally homogeneous spaces}
\author{Marco Bramanti and Maochun Zhu}
\maketitle

\begin{abstract}
We introduce the concept of locally homogeneous space, and prove in this
context $L^{p}$ and $C^{\alpha}$ estimates for singular and fractional
integrals, as well as $L^{p}$ estimates on the commutator of a singular or
fractional integral with a $BMO$ or $VMO$ function. These results are
motivated by local a-priori estimates for subelliptic equations.

\end{abstract}

\section{Introduction}

\subsubsection*{Motivation}

The theory of singular integrals has been usefully applied to local a priori
estimates for PDEs in several contexts of increasing generality, in the last
decades. The abstract framework of spaces of homogeneous type, introduced by
Coifman-Weiss in \cite{CW}, has proved to be a suitable framework in many
cases, so far: we have a set (which in concrete applications is usually a
bounded domain of $\mathbb{R}^{n}$), a distance or a quasidistance adapted to
the differential operator (the Euclidean distance for classical elliptic
equations, parabolic distance for parabolic equations, Carnot-Carath\'{e}odory
distance -or some variation of it- for operators built on H\"{o}rmander's
vector fields -see \cite{NSW}-, and so on), and a measure (usually the
Lebesgue measure) which is \emph{doubling} with respect to the metric balls.
In these situations the quasidistance $\rho$ is usually defined in some
$\Omega_{0}$ which is either the whole $\mathbb{R}^{n}$ or some domain which
is larger than the bounded domain $\Omega$ where we want to prove our
estimates. Since the balls $B\left(  x,r\right)  $ are, by definitions,
subsets of $\Omega_{0}$, that is%
\[
B\left(  x,r\right)  =\left\{  y\in\Omega_{0}:\rho\left(  x,y\right)
<r\right\}  ,
\]
if we want to apply the theory of spaces of homogeneous type to the set
$\Omega,$ the doubling condition we have to check is%
\begin{equation}
\mu\left(  B\left(  x,2r\right)  \cap\Omega\right)  \leq c\mu\left(  B\left(
x,r\right)  \cap\Omega\right)  \text{ for any }x\in\Omega,r>0.
\label{doubling 1}%
\end{equation}
On the other hand, the doubling condition that reasonable $\rho$ and $\mu$
usually satisfy is%
\begin{equation}
\mu\left(  B\left(  x,2r\right)  \right)  \leq c\mu\left(  B\left(
x,r\right)  \right)  \text{ for any }x\in\Omega,0<r<r_{0} \label{doubling 2}%
\end{equation}
for some small $r_{0}$. Passing from (\ref{doubling 2}) to (\ref{doubling 1})
requires some \textquotedblleft smoothness\textquotedblright\ property of
$\partial\Omega,$ a property which, however, is not a natural requirement for
our original local problem, but more a technical complication due to the fact
that, in order to apply the theory of spaces of homogeneous type, we are
regarding the set $\Omega,$ which in our problem is a subset of a larger
universe, as the universe itself. If $\mu$ is the Lebesgue measure and $\rho$
is the Euclidean distance, in order to fulfil (\ref{doubling 1}) it is enough
to require $\partial\Omega$ Lipschitz; if $\rho$ is the
Carnot-Carath\'{e}odory distance induced by a set of H\"{o}rmander's vector
fields $X_{1},X_{2},...,X_{q}$ which is naturally attached to the study of the
operator
\begin{equation}
\sum_{i=1}^{q}X_{i}^{2} \label{Hormander squares}%
\end{equation}
then (\ref{doubling 1}) is satisfied for instance when $\Omega$ is itself a
metric ball, as was first proved in \cite{FL} (see also \cite[Lemma 4.2]{BB3}
for a more self-contained proof); this result basically relies on the fact
that this quasidistance has the \emph{segment property}, which essentially
means that for any couple of points $x_{1},x_{2}$ at distance $r$ and for any
number $\delta<r$ there exists a point $x_{0}$ having distance $\delta$ from
$x_{1}$ and $r-\delta$ from $x_{2}$. An analog result of regularity for the
metric ball has been proved in \cite{BB4} for the \textquotedblleft parabolic
Carnot-Carath\'{e}odory distance\textquotedblright\ attached to the operator
\[
\partial_{t}-\sum_{i=1}^{q}X_{i}^{2}.
\]
This distance has no longer the segment property, but the simple way in which
the time variable is involved allows to do explicit (but nontrivial!)
computations and show that when $\Omega$ is a metric ball, (\ref{doubling 1})
still holds.

If we now pass to consider H\"{o}rmander's operators of the kind%
\begin{equation}
\sum_{i=1}^{q}X_{i}^{2}+X_{0} \label{Hormander drift}%
\end{equation}
(where the drift term $X_{0}$ is part of the set $X_{0},X_{1},...,X_{q}$ which
satisfies H\"{o}rmander's condition), the corresponding quasidistance $\rho$
is the one defined by Nagel-Stein-Wainger in \cite{NSW}; this quasidistance
has been much less studied than the usual Carnot-Carath\'{e}odory distance (we
can quote, in the context of nonsmooth H\"{o}rmander's vector fields, the two
papers \cite{BBP1}, \cite{BBP2}). Although a local doubling condition
(\ref{doubling 2}) holds, as proved in \cite{NSW}, this quasidistance does not
satisfy the segment property and, as far as we know, a condition of kind
(\ref{doubling 1}) has never been proved for $\Omega$ a metric ball, or for
any other special kind of bounded domain $\Omega.$ Therefore the existing
results do not allow to apply the theory of spaces of homogeneous type to the
space $\left(  \Omega,\rho,\mu\right)  $ when $\Omega$ is some bounded domain
of $\mathbb{R}^{n},$ $\rho$ is the Nagel-Stein-Wainger distance attached to
the set of H\"{o}rmander's vector fields $X_{1},X_{2},...,X_{0}$ (with $X_{0}$
of \textquotedblleft weight\textquotedblright\ two), and $\mu$ the Lebesgue
measure. This problem has been sometimes overlooked, apparently; for instance,
in the famous paper \cite{RS}, $L^{p}$ estimates are proved for operators
(\ref{Hormander squares}), and stated for operators (\ref{Hormander drift}),
without any reference to the mere existence of the problem of assuring the
validity of condition (\ref{doubling 1}). On the other hand, as already
suggested, one feels that if our final goal is that of proving local
estimates, no kind of \textquotedblleft smoothness\textquotedblright\ of the
domain $\Omega$ with respect to the quasidistance should be a crucial
requirement; in other words, it is reasonable that this difficulty could be
bypassed. \emph{The basic scope of the present paper is to build up a local
theory of singular integrals which does not require checking condition
(\ref{doubling 1}), when (\ref{doubling 2}) is known.}

Another problem of a-priori estimates for PDEs in which proving that
(\ref{doubling 2}) implies (\ref{doubling 1}) for some domain $\Omega$ appears
troublesome has been studied in \cite{BCLP}. In that paper the Authors
consider a class of Kolmogorov-Fokker-Planck operators on $\mathbb{R}%
^{n}\times\left[  -1,1\right]  ,$ for which the natural\ quasidistance is a
function $\rho$ (not equivalent to the Carnot-Carath\'{e}odory distance
induced by a system of H\"{o}rmander's vector fields), which satisfies the
quasi-triangle inequality on any compact set and satisfies a local doubling
condition (\ref{doubling 2}) for any bounded $\Omega$; again, however, one has
no idea of how to prove (\ref{doubling 1}) for some particular bounded
$\Omega.$ In that case, the Authors overcame the problem by applying an ad-hoc
theory of singular integrals in \emph{nondoubling }spaces, developed in
\cite{B2}. The resort to theories of singular integrals in nondoubling
contexts, as have been developed in the last decade by Tolsa,
Nazarov-Treil-Volberg, and other authors (see for instance the book \cite{V}
and references therein), is actually an alternative possibility in order to
bypass (\ref{doubling 1}). However, and here another motivation of the present
paper comes in, when proving a-priori $L^{p}$ estimates for PDEs with $VMO$
coefficients (a line of research which started with the papers by
Chiarenza-Frasca-Longo \cite{CFL1}, \cite{CFL2} and developed in several
directions in the last 20 years), one needs a series of results about
commutators of singular and fractional integrals with $BMO$ functions, which
do not have a natural counterpart in the nondoubling context; more precisely,
results of this kind have been actually proved by Tolsa \cite{T}, in the
context of $\left(  \mathbb{R}^{n},d,d\mu\right)  $, where $\mu$ is a very
general Radon measure, but $d$ is the Euclidean distance. Since the extension
of these deep results to the case of a general quasidistance $d$ appears far
from being obvious, it seems easier and more natural for the problem under
exam to establish these commutator theorems in the framework of a theory of
singular integrals in a locally doubling context (instead than in a
nondoubling one). More generally, we think that the idea of proving a local
version of some basic results about singular integrals is a very natural one,
and we feel that these results can be of some interest also for other applications.

\subsubsection*{Main results}

In this paper we will prove, in the context of \emph{locally homogeneous
spaces }(which will be defined in the next section), results of continuity, in
$L^{p}$ and $C^{\alpha}$ spaces, for singular and fractional integrals, as
well as $L^{p}$ estimates for the commutator of a singular or fractional
integral with the multiplication with a $BMO$ function. Also, we will state
these commutator theorems in a form suitable to prove the smallness of the
$L^{p}$ norm of the commutator on a small ball, whenever the $BMO$ function
actually belongs to $VMO.$ This localized version of the commutator theorem
for a $VMO$ function, in the original Euclidean case (exploited in
\cite{CFL1}, \cite{CFL2}) relies on the possibility of approximating a $VMO$
function by a uniformly continuous function in $BMO$ norm, and on the
possibility of extending to the whole space a uniformly continuous function
defined on a ball, preserving the continuity modulus. Here we manage to
establish directly the localized version of the commutator theorems for a
$VMO$ function, without the necessity of proving the aforementioned
approximation and extension results. Therefore, under this respect, our
approach conceptually simplifies also the Euclidean case. We will also deal,
very briefly, with a local version of the maximal function and its $L^{p}$
continuity, another tool which is useful in the alluded applications.

Our main results are: Theorems \ref{Theorem L^p C^eta} and \ref{frac lp-lq} in
Section \ref{section singular}; Theorems \ref{Thm commutator},
\ref{Thm comm frac}, \ref{Thm comm pos} in Section \ref{Section Commutator};
Theorem \ref{Thm maximal} in Section \ref{Section Maximal}. We also think that
the basic theory developed in Sections \ref{section abstract}%
-\ref{Section Dyadic}, particularly Theorems \ref{Main Thm} and \ref{Thm F},
could be useful to prove further results in the same spirit.

All the results proved in this paper will be used in the proof of $L^{p}$ and
$C^{\alpha}$ estimates for nonvariational operators structured on
H\"{o}rmander's operators of type (\ref{Hormander drift}), that is for
operators of the form%
\[
\sum_{i.j=1}^{q}a_{ij}\left(  x\right)  X_{i}X_{j}+a_{0}\left(  x\right)
X_{0}%
\]
where $X_{0},X_{1},...,X_{q}$ are a system of smooth H\"{o}rmander's vector
fields in a bounded domain of $\mathbb{R}^{n}$ ($n>q+1$), $\left\{
a_{ij}\right\}  $ is a uniformly positive matrix with bounded entries, $a_{0}$
is bounded and bounded away from zero and all the coefficients $a_{ij},a_{0}$
belong to the suitable function space $VMO$ or $C^{\alpha}$ (respectively, to
prove $L^{p}$ and $C^{\alpha}$ estimates on $X_{i}X_{j}u$).\ These estimates
are proved in a separate paper \cite{BZ}, and generalize the results proved in
\cite{bb2} and \cite{BB4} when the drift $X_{0}$ is lacking.

\subsubsection*{Strategy}

The basic idea, in order to bypass the necessity of checking condition
(\ref{doubling 1}) instead of (\ref{doubling 2}), is to adapt the abstract
construction of dyadic cubes in spaces of homogeneous type carried out by
Christ in \cite{C}. In that paper, the Author shows how in any space of
homogeneous type one can construct, for any $k\in\mathbb{Z}$, a family of
\textquotedblleft dyadic cubes\textquotedblright\ of diameter comparable to
$\delta^{k}$ (with $\delta$ a small fixed number).\ Actually these
\textquotedblleft cubes\textquotedblright\ are open sets, defined by an
abstract construction, which nevertheless share with the classical dyadic
cubes all the basic properties. The relevant fact for us is that each of these
cubes $Q$ is, in turn, a space of homogeneous type, that is in \cite{C}\ the
Author proves that (\ref{doubling 1}) actually holds for $\Omega=Q$. Here we
adapt the previous construction in a local setting. Given our space $\Omega,$
which is seen as the union of an increasing sequence of bounded subsets
$\Omega_{n},$ we construct for each $n$ and each scale $k=1,2,3...$, a family
of (small) dyadic cubes essentially covering $\Omega_{n}$ and contained in
$\Omega_{n+1};$ each of these cubes can still be proved to be a space of
homogeneous type; moreover, the same is true for any finite union of dyadic
cubes of the same scale $k$. The idea is then to apply known results for
spaces of homogeneous type to suitable unions of dyadic cubes which cover a
fixed small ball, and derive the corresponding result on the ball. Since
dyadic cubes are abstract objects, which in concrete applications of the
theory cannot be explicitly exhibited, our job is to use dyadic cubes just as
a tool, but to state and prove all our results in the language of balls, to
make them easily applicable.

To make more transparent the strategy of our construction, let us point out
what follows.

We will show that for any $n$ we can cover $\Omega_{n}$ with a finite union of
balls of any prescribed small size, and for each of these balls we can
construct a space of homogeneous type $F$ which is contained in $\Omega
_{n+1},$ \textquotedblleft almost contains\textquotedblright\ this ball $B$,
and is comparable to $B$, both in measure and in diameter. This
\textquotedblleft almost inclusion\textquotedblright\ is made precise in two ways:

1) $F\supset B\setminus E$ where $E$ is a zero measure set; this inclusion is
enough to handle $L^{p}$ estimates or more generally estimates which involve
integral norms;

2) the closure of $F$ contains $B$; this inclusion is enough to handle
$C^{\alpha}$ estimates, or more generally estimates which involve moduli of
uniform continuity of the functions.

The idea of exploiting Christ's construction of dyadic cubes to prove results
in a locally doubling context has been already used by Carbonaro, Mauceri,
Meda in \cite{M}; their context, however, is different from ours: in that
paper the Authors consider a situation where the measure of balls grows fast
at infinity, so that the doubling condition holds for balls of radius $r\leq
r_{0},$ for any fixed $r_{0}$; on the other hand, these Authors have not our
problem of keeping far from the boundary of a bounded domain, to avoid
intersections. Moreover, they use dyadic cubes to adapt the proofs which hold
in the doubling case, while our strategy is not to adapt the existing proofs
but to apply the existing results which hold in the doubling case.

The construction of a suitable family of spaces of homogeneous type is not the
only problem to overcome in our situation. The possible overlapping of the
balls $B\left(  x,r\right)  $ with the boundary of the considered domain
creates problems under at least other two regards. The first is the validity
of a suitable cancellation property for the kernel of singular integral
operator: if we know that a singular kernel has a bounded integral over small
spherical shells, this does not imply the persistence of this property when we
integrate over the intersection of the spherical shell with some fixed domain.
This problem will be solved using suitable cutoff functions. The second
problem is to suitably define and handle $BMO$ and $VMO$ spaces, avoiding to
take the average of a function over the intersection of $B\left(  x,r\right)
$ with a fixed domain. To this aim we will introduce a $BMO_{loc}$ space
adapted to a couple of domains $\Omega_{n}\subset\Omega_{n+1},$ which in our
context is a natural notion, and we will show which relation this space has
with the standard $BMO$.

\subsubsection*{Plan of the paper}

In Section \ref{section abstract} we state precisely our definition of locally
homogeneous spaces and draw some first consequences of the definition, in
terms of topology and measure. Section \ref{Section Dyadic} contains the
construction of dyadic cubes and the proof of their relevant properties which
will allow to apply the theory of spaces of homogeneous type. In Section
\ref{Section Holder} we build, in a fairly standard way, H\"{o}lder continuous
cutoff functions, another tool which will be useful in the following. In
Section \ref{section singular} we prove our local $L^{p}$ and $C^{\alpha}$
continuity results for singular and fractional integrals. In Section
\ref{Section BMO} we introduce $BMO$ and $VMO$ spaces, both in the standard
and in a local version, and study the relation between the two concepts. In
Section \ref{Section Commutator} we prove local $L^{p}$ estimates on the
commutator of a singular or fractional integral with a $BMO$ or $VMO$
function. In Section \ref{Section Maximal} we deal with the local maximal
operator and its $L^{p}$ continuity. Finally, in Section
\ref{section quasisymmetric} we extend the results of Sections
\ref{section singular} to \ref{Section Maximal} to the more general situation
where the local quasidistance is quasisymmetric but not symmetric. An Appendix
collects all the known results about spaces of homogeneous type which we need
throughout the paper.

\bigskip

\textbf{Acknowledgements.} This research was mainly carried out while Maochun
Zhu was visiting the Department of Mathematics of Politecnico di Milano, which
we wish to thank for the hospitality. The project was supported by the
National Natural Science Foundation of China (Grant No. 10871157), Specialized
Research Fund for the Doctoral Program of Higher Education (No. 200806990032).

\section{The abstract framework of locally homogeneous
spaces\label{section abstract}}

We are going to state the assumptions which will define the notion of locally
homogeneous space. (For comparison, the standard definition of space of
homogeneous type is recalled in the Appendix, Section \ref{section appendix}).

(H1) Let $\Omega$ be a set, endowed with a function $\rho:\Omega\times
\Omega\rightarrow\lbrack0,\infty)$ such that for any $x,y\in\Omega$:

(a) $\rho\left(  x,y\right)  =0\Leftrightarrow x=y;$

(b) $\rho\left(  x,y\right)  =\rho\left(  y,x\right)  .$

For any $x\in\Omega,r>0,$ let us define the ball%
\[
B\left(  x,r\right)  =\left\{  y\in\Omega:\rho\left(  x,y\right)  <r\right\}
.
\]
These balls can be used to define a topology in $\Omega$, saying that
$A\subset\Omega$ is open if for any $x\in A$ there exists $r>0$ such that
$B\left(  x,r\right)  \subset A$. Also, we will say that $E\subset\Omega$ is
bounded if $E$ is contained in some ball.

Let us assume that:

(H2) (a) the balls are open with respect to this topology;

(H2) (b) for any $x\in\Omega$ and $r>0$ the closure of $B\left(  x,r\right)  $
is contained in $\left\{  y\in\Omega:\rho\left(  x,y\right)  \leq r\right\}
.$

We will prove in a moment that the validity of conditions (H2) (a) and (b) is
equivalent to the following:

(H2') $\rho\left(  x,y\right)  $ is a continuous function of $x$ for any fixed
$y\in\Omega.$

(H3) Let $\mu$ be a positive regular Borel measure in $\Omega.$

(H4) Assume there exists an increasing sequence $\left\{  \Omega_{n}\right\}
_{n=1}^{\infty}$ of bounded measurable subsets of $\Omega,$ such that:%

\begin{equation}%
{\displaystyle\bigcup_{n=1}^{\infty}}
\Omega_{n}=\Omega\label{Hp 0}%
\end{equation}
and such for, any $n=1,2,3,...$:

(i) the closure of $\Omega_{n}$ in $\Omega$ is compact;

(ii) there exists $\varepsilon_{n}>0$ such that%
\begin{equation}
\left\{  x\in\Omega:\rho\left(  x,y\right)  <2\varepsilon_{n}\text{ for some
}y\in\Omega_{n}\right\}  \subset\Omega_{n+1}; \label{Hp 1}%
\end{equation}
\qquad

(H5) there exists $B_{n}\geq1$ such that for any $x,y,z\in\Omega_{n}$%
\begin{equation}
\rho\left(  x,y\right)  \leq B_{n}\left(  \rho\left(  x,z\right)  +\rho\left(
z,y\right)  \right)  ; \label{Hp 2}%
\end{equation}

(H6) there exists $C_{n}>1$ such that for any $x\in\Omega_{n},0<r\leq
\varepsilon_{n}$ we have%
\begin{equation}
\text{ }0<\mu\left(  B\left(  x,2r\right)  \right)  \leq C_{n}\mu\left(
B\left(  x,r\right)  \right)  <\infty. \label{Hp 3}%
\end{equation}
(Note that for $x\in\Omega_{n}$ and $r\leq\varepsilon_{n}$ we also have
$B\left(  x,2r\right)  \subset\Omega_{n+1}$).

\begin{definition}
\label{Def loc hom space}We will say that $\left(  \Omega,\left\{  \Omega
_{n}\right\}  _{n=1}^{\infty},\rho,\mu\right)  $ is a \emph{locally
homogeneous space }if assumptions (H1) to (H6) hold.
\end{definition}

\textbf{Dependence on the constants. }The numbers $\varepsilon_{n},B_{n}%
,C_{n}$ will be called \textquotedblleft the constants of $\Omega_{n}%
$\textquotedblright. It is not restrictive to assume that $B_{n},C_{n}$ are
nondecreasing sequences, and $\varepsilon_{n}$ is a nonincreasing sequence.
Throughout the paper our estimates, for a fixed $\Omega_{n},$ will often
depend not only on the constants of $\Omega_{n},$ but also (possibly) on the
constants of $\Omega_{n+1},\Omega_{n+2},\Omega_{n+3}$. We will briefly say
that \textquotedblleft a constant depends on $n$\textquotedblright\ to mean
this type of dependence.

\bigskip

In the language of \cite{CW}, $\rho$ is a \emph{quasidistance} in each set
$\Omega_{n};$ we can also say that it is a \emph{local quasidistance in}
$\Omega.$ We stress that the two conditions appearing in (H2) are logically
independent each from the other, and they do not follow from (\ref{Hp 2}),
even when $\rho$ is a quasidistance\emph{ }in $\Omega,$ that is when
$B_{n}=B>1$ for any $n$. If, however, $\rho$ is a \emph{distance }in $\Omega$,
that is $B_{n}=1$ for any $n$, then (H2) is automatically fulfilled.

The continuity of $\rho$ also implies that (\ref{Hp 2}) still holds for
$x,y,z\in\overline{\Omega}_{n}.$ We will sometimes exploit this fact.

Also, note that $\mu\left(  \Omega_{n}\right)  <\infty$ for every $n$, since
$\overline{\Omega}_{n}$ is compact. (This follows by the regularity of $\mu,$
or also from the finiteness of the measure of balls, see (H6), since
$\overline{\Omega}_{n}$ can be covered by a finite number of small balls).

We also point out that assuming $\mu$ \emph{regular }(see (H3)) is not really
necessary, as will be explained after Proposition \ref{Prop density}; however,
since this assumption is harmless in the applications we are interested in, we
prefer to keep it, in order to avoid the necessity of entering into annoying details.

\begin{example}
(i)\ If $\left(  \Omega,\rho,d\mu\right)  $ is a bounded space of homogeneous
type in the sense of Coifman-Weiss \cite{CW} (the definition will be recalled
in the Appendix), the above conditions are fulfilled choosing $\Omega
_{n}\equiv\Omega$ and $B_{n}=B\geq1$ for any $n$.

(ii) In the applications to subelliptic equations that we have in mind, and
will be dealt in \cite{BZ}, $\Omega$ is a bounded domain of $\mathbb{R}^{N}$
and $\Omega_{n}$ an increasing sequence of bounded domains, with $\Omega
_{n}\Subset\Omega_{n+1}\Subset\Omega$ for any $n$; $\rho$ is the
Nagel-Stein-Wainger distance induced by a family $X_{0},X_{1},X_{2},...,X_{q}$
of H\"{o}rmander's vector fields, where $X_{0}$ has weight two, $\mu$ the
Lebesgue measure in $\mathbb{R}^{N}$.

(iii) The same setting of (ii) fits the theory of \emph{nonsmooth
H\"{o}rmander's vector fields}, as dealt in \cite{BBP1}, \cite{BBP2}.

Note that in the situations (ii)-(iii) $\rho$ is actually a distance, which
induces the Euclidean topology, and (H6) is a known result.

(iv) In the situation considered in \cite{BCLP}, $\Omega=\mathbb{R}^{N}%
\times\left[  -1,1\right]  ,\Omega_{n}=B_{n}\times\left[  -1,1\right]  $ where
$B_{n}$ is the Euclidean ball of center $0$ and radius $n$ in $\mathbb{R}%
^{N},\rho\left(  z,\zeta\right)  =\left\Vert \zeta^{-1}\circ z\right\Vert $
where $\circ$ is a Lie group operation related to the differential operator
\[
L=\sum_{i,j=1}^{p_{0}}a_{ij}\partial_{x_{i}x_{j}}^{2}+\sum_{i,j=1}^{N}%
b_{ij}x_{i}\partial_{x_{j}}-\partial_{t}\text{ \ \ }\left(  \text{where }%
p_{0}<N\right)
\]
which is under study, while $\left\Vert \cdot\right\Vert $ is the homogeneous
norm defined by the family of dilations related to another differential
operator, which is the \textquotedblleft principal part\textquotedblright\ of
$L.$ Therefore $\rho$ is neither the usual distance considered in Carnot
groups, nor is (equivalent to) the Carnot-Carath\'{e}odory distance induced by
the vector fields; $\rho$ satisfies (H5) and induces the Euclidean topology;
from its analytical definition it is clear that $\rho$ is continuous, hence
(H2) is fulfilled; also (H6) can be proved. More precisely, the function
$\rho$ considered in \cite{BCLP} is not symmetric but satisfies a weaker
condition: for any $n$ there exists $A_{n}>1$ such that%
\begin{equation}
\rho\left(  x,y\right)  \leq A_{n}\rho\left(  y,x\right)  \text{ for any
}x,y\in\Omega_{n}. \label{quasisymmetry}%
\end{equation}
This motivates a further extension of our theory, as we will explain in a moment.
\end{example}

\begin{remark}
[Extension to quasisymmetric functions]As anticipated in the above Example
(iv), in view of some applications it is desirable to consider the more
general setting in which $\rho$ is not assumed symmetric, but satisfies
condition (\ref{quasisymmetry}). However, developing the whole theory of
Section \ref{Section Dyadic} under this weaker assumption would make our
computations considerably heavier. Instead, it is much easier to develop first
the theory under the symmetry assumption, and then to show that our main
results about singular and fractional integrals still hold if we replace the
symmetry condition with (\ref{quasisymmetry}). This extension will be
discussed in Section \ref{section quasisymmetric}.
\end{remark}

In the rest of this section we will make some remarks and prove some easy
facts related to topology and measure in a locally homogeneous space.

Since, by (H2) (a), the balls are open, for each $x\in\Omega$ the balls
$B\left(  x,r\right)  $ satisfy the axioms of complete system of neighborhoods
of $x;$ hence the topological space $\Omega$ is first countable, and
continuity and closedness can be discussed by means of sequences of points.
Let us prove the following fact, that we have claimed before.

\begin{proposition}
\label{Prop topology}Conditions (H2) are equivalent to condition (H2').
\end{proposition}

\begin{proof}
Assume (H2), and let us prove the continuity of $x\longmapsto\rho\left(
x,y\right)  $. Fix $x\in\Omega$ and take a sequence $\left\{  x_{k}\right\}  $
converging to $x$. Let us show that $\rho\left(  x_{k},y\right)
\rightarrow\rho\left(  x,y\right)  $ for any $y\in\Omega$. Let $r=\rho\left(
x,y\right)  $ and $\varepsilon>0$; since $x\in B\left(  y,r+\varepsilon
\right)  $ and the balls are open, there exists $B\left(  x,\delta\right)
\subset B\left(  y,r+\varepsilon\right)  .$ Then $x_{k}\in B\left(
y,r+\varepsilon\right)  $ definitively, since $x_{k}\rightarrow x.$ This
implies that $\rho\left(  x_{k},y\right)  <r+\varepsilon$ definitively, so
that%
\[
\text{limsup}_{k\rightarrow\infty}\rho\left(  x_{k},y\right)  \leq
r+\varepsilon.
\]
This holds for any $\varepsilon>0,$ hence%
\[
\text{limsup}_{k\rightarrow\infty}\rho\left(  x_{k},y\right)  \leq r.
\]
We now want to show that%
\begin{equation}
\text{liminf}_{k\rightarrow\infty}\rho\left(  x_{k},y\right)  \geq r,
\label{liminf}%
\end{equation}
which will imply $\rho\left(  x_{k},y\right)  \rightarrow\rho\left(
x,y\right)  $. Let again $r=\rho\left(  x,y\right)  $ and $\varepsilon>0$;
then $x\notin B\left(  y,r-\frac{\varepsilon}{2}\right)  .$ By (H2) (b),
\[
\overline{B\left(  y,r-\varepsilon\right)  }\subset B\left(  y,r-\frac
{\varepsilon}{2}\right)  ,
\]
hence (denoting the complement of $A$ with $A^{c}$) $x$ belongs to
$\overline{B\left(  y,r-\varepsilon\right)  }^{c}$; since this is an open set,
there exists $B\left(  x,\eta\right)  \subset\overline{B\left(
y,r-\varepsilon\right)  }^{c};$ hence $x_{k}\in\overline{B\left(
y,r-\varepsilon\right)  }^{c}$ definitively, which means that $\rho\left(
x_{k},y\right)  \geq r-\varepsilon$ definitively, and
\[
\text{liminf}_{k\rightarrow\infty}\rho\left(  x_{k},y\right)  \geq
r-\varepsilon.
\]
This holds for any $\varepsilon>0,$ so (\ref{liminf}) follows.

Conversely, assume now the continuity of $\rho,$ and let us prove (H2). Let
$y\in B\left(  x,r\right)  ,$ so that $\rho\left(  x,y\right)  <r.$ Since
$\rho$ is continuous, there exists $B\left(  y,r^{\prime}\right)  $ such that
for any $z\in B\left(  y,r^{\prime}\right)  $ we have $\rho\left(  x,z\right)
<r;$ hence $B\left(  y,r^{\prime}\right)  \subset B\left(  x,r\right)  $ and
$B\left(  x,r\right)  $ is open, that is (H2) (a)\ holds.

Let now $y\in\overline{B\left(  x,r\right)  }$; since we already know that
balls are open, as noted above this means that $y_{k}\rightarrow y$ for some
sequence $\left\{  y_{k}\right\}  \subset B\left(  x,r\right)  .$ Hence
$\rho\left(  y_{k},x\right)  <r$ and limsup$_{k\rightarrow\infty}\rho\left(
y_{k},x\right)  \leq r.$ However, $\rho$ is continuous, so
\[
\text{limsup}_{k\rightarrow\infty}\rho\left(  y_{k},x\right)  =\rho\left(
y,x\right)  ,
\]
which means that $y\in\left\{  z:\rho\left(  z,x\right)  \leq r\right\}  ,$
which is (H2) (b).
\end{proof}

The next property, which involves both $\rho$ and the measure $\mu,$ tells us
that also even when estimating H\"{o}lder norms, zero measure sets are negligible.

\begin{proposition}
\label{Holder negligible}(i) Let $A,E\subset\Omega,$ $A$ open and $E$ of
measure zero. Then $\overline{A\setminus E}=\overline{A}.$

(ii) Let $f:A\setminus E\rightarrow\mathbb{R}$ with $A,E$ as above, and $f$
such that, for some $\alpha,C>0$%
\begin{equation}
\left\vert f\left(  x\right)  -f\left(  y\right)  \right\vert \leq
C\rho\left(  x,y\right)  ^{\alpha} \label{holder}%
\end{equation}
for any $x,y\in A\setminus E.$ Then $f$ can be continuously extended to
$\overline{A}$ in such a way that (\ref{holder}) holds for any $x,y\in
\overline{A}.$
\end{proposition}

\begin{proof}
(i) Let $x\in\overline{A}$ and $\left\{  x_{k}\right\}  \subset A$ such that
$x_{k}\rightarrow x.$ Since $A$ is open, for any $k$ there exists $r_{k}>0$
such that $B\left(  x_{k},r_{k}\right)  \subset A.$ It is not restrictive to
assume that $r_{k}\rightarrow0.$ For any $k$ there exists $y_{k}\in B\left(
x_{k},r_{k}\right)  \setminus E$; otherwise $E$ would contain a ball, which by
(H6) has positive measure. By (\ref{Hp 0}), $x\in\Omega_{n}$ for some $n$;
then, by (\ref{Hp 1}) the sequence $\left\{  x_{k}\right\}  $ is definitively
contained in $\Omega_{n+1};$ for the same reason the balls $B\left(
x_{k},r_{k}\right)  $ are definitively contained in $\Omega_{n+2},$ hence we
can apply the quasitriangle inequality (\ref{Hp 2}) writing%
\[
\rho\left(  y_{k},x\right)  \leq B_{n+2}\left(  \rho\left(  y_{k}%
,x_{k}\right)  +\rho\left(  x_{k},x\right)  \right)  \leq B_{n+2}\left(
r_{k}+\rho\left(  x_{k},x\right)  \right)  \rightarrow0
\]
for $k\rightarrow\infty,$ so $y_{k}\rightarrow x.$ Since $y_{k}\in A\setminus
E$, this implies $x\in\overline{A\setminus E}$ and we are done.

(ii) Since $\rho$ is continuous, (\ref{holder}) implies that $f$ is uniformly
continuous on $A\setminus E,$ hence it can be continuously extended to
$\overline{A\setminus E}$ in such a way that (\ref{holder}) still holds. By
point (i) $\overline{A\setminus E}=\overline{A}$, so (ii) is proved.
\end{proof}

\section{Dyadic cubes in a locally homogeneous space\label{Section Dyadic}}

Throughout this paper, until Section 8, we will assume that $\left(
\Omega,\left\{  \Omega_{n}\right\}  _{n=1}^{\infty},\rho,\mu\right)  $ be a
locally homogeneous space.

The construction of dyadic cubes, which has been anticipated in the
introduction, is contained in the following:

\begin{theorem}
\label{Main Thm}For any $n=1,2,3,...\ $there exists a collection of open sets
\[
\left\{  Q_{\alpha}^{k}\subset\Omega,k=1,2,3...,\alpha\in I_{k}\right\}
\]
(where $I_{k}$ is a set of indices), positive constants $a_{0},c_{0}%
,c_{1},c_{2},\delta\in\left(  0,1\right)  $ and a set $E\subset\Omega_{n}$ of
zero measure, such that for any $k=1,2,3...$ we have:

(a) $\forall\alpha\in I_{k},$ each $Q_{\alpha}^{k}$ contains a ball $B\left(
z_{\alpha}^{k},a_{0}\delta^{k}\right)  ;$

(b) $%
{\displaystyle\bigcup_{\alpha\in I_{k}}}
Q_{\alpha}^{k}\subset\Omega_{n+1};$

(c) $\forall\alpha\in I_{k},1\leq l\leq k$ there exists $Q_{\beta}%
^{l}\supseteq Q_{\alpha}^{k};$

(d) $\forall\alpha\in I_{k},$ diam$\left(  Q_{\alpha}^{k}\right)  <c_{1}%
\delta^{k}$ and $\overline{Q_{\alpha}^{k}}\subset B\left(  z_{\alpha}%
^{k},c_{1}\delta^{k}\right)  ;$

(e) $\ell\geq k\Longrightarrow\forall\alpha\in I_{k},\beta\in I_{l},Q_{\beta
}^{\ell}\subset Q_{\alpha}^{k}$ or $Q_{\beta}^{\ell}\cap Q_{\alpha}%
^{k}=\emptyset;$

(f) $\Omega_{n}\setminus%
{\displaystyle\bigcup_{\alpha\in I_{k}}}
Q_{\alpha}^{k}\subset E;$

(g) $\forall\alpha\in I_{k},$ $x\in Q_{\alpha}^{k}\setminus E,$ $j\geq1$ there
exists $Q_{\beta}^{j}\ni x;$

(h) $\mu\left(  B\left(  x,2r\right)  \cap Q_{\alpha}^{k}\right)  \leq
c_{2}\mu\left(  B\left(  x,r\right)  \cap Q_{\alpha}^{k}\right)  $ for any
$x\in Q_{\alpha}^{k}\setminus E,r>0.$ More precisely, for these $x$ and $r$ we
have:%
\begin{equation}
\mu\left(  B\left(  x,r\right)  \cap Q_{\alpha}^{k}\right)  \geq\left\{
\begin{tabular}
[c]{ll}%
$c_{0}\mu\left(  B\left(  x,r\right)  \right)  $ & $\text{for }r\leq\delta
^{k}$\\
$c_{0}\mu\left(  Q_{\alpha}^{k}\right)  $ & $\text{for }r>\delta^{k}$%
\end{tabular}
\ \ \ \ \ \ \ \right.  \label{lower bounds cubes}%
\end{equation}

\end{theorem}

Note that the cubes $Q_{\alpha}^{k}$ and all the constants depend on $n,$ so
we should write, more precisely%
\[
\left\{  Q_{\alpha}^{\left(  n\right)  ,k}\right\}  _{\alpha\in I_{k}^{\left(
n\right)  }};\delta_{\left(  n\right)  };a_{0,\left(  n\right)  },c_{0,\left(
n\right)  },c_{1,\left(  n\right)  },c_{2,\left(  n\right)  }.
\]
However, in order to simplify notation, we will skip the index $\left(
n\right)  $ whenever doing so does not create ambiguity. As will be apparent
from the proof, the sequence of constants $\delta_{\left(  n\right)  }$ can be
assumed nonincreasing.

The sets $Q_{\alpha}^{k}$ can be thought as dyadic cubes of sidelength
$\delta^{k}$. Note that $k$ is a \emph{positive }integer, so we are only
considering \emph{small }dyadic cubes.

The proof of Theorem \ref{Main Thm} is not much more than a careful inspection
and adaptation of some proofs contained in \cite{C}. However, our iterative
construction is a bit more involved because at every step $n$ the
\textquotedblleft universe\textquotedblright\ that we want to cover with our
cubes enlarges. Moreover, in order to use the quasitriangle inequality, we
need to know in advance that the points belong to some domain; this will be
often proved by a tricky combined use of (\ref{Hp 1}) and (\ref{Hp 2}).

\begin{proof}
[Proof of Theorem \ref{Main Thm}, first part]For a fixed $\Omega_{n},$ let
$\delta>0$ to be fixed later, and let us perform the following iterative construction.

For $k=1$, let us fix a maximal collection of points $\left\{  z_{\alpha}%
^{1}\right\}  _{\alpha\in I_{1}}\subset\Omega_{n}$ such that%
\[
\rho\left(  z_{\alpha}^{k},z_{\beta}^{k}\right)  \geq\delta\text{ for any
}\alpha\neq\beta.
\]
By the maximality, we can say that for $x\in\Omega_{n}$ there exists
$z_{\alpha}^{1}$ such that $\rho\left(  z_{\alpha}^{1},x\right)  <\delta,$
hence%
\[
E_{1}\equiv\Omega_{n}\subseteq%
{\displaystyle\bigcup\limits_{\alpha\in I_{1}}}
B\left(  z_{\alpha}^{1},\delta\right)  \equiv E_{2}.
\]
For $k=2,$ let us fix a maximal collection of points $\left\{  z_{\alpha}%
^{2}\right\}  _{\alpha\in I_{2}}\subset E_{2}$ such that%
\[
\rho\left(  z_{\alpha}^{2},z_{\beta}^{2}\right)  \geq\delta^{2}\text{ for any
}\alpha\neq\beta.
\]
By the maximality, we can say that for any $x\in E_{2}$ there exists
$z_{\alpha}^{2}$ such that $\rho\left(  z_{\alpha}^{1},x\right)  <\delta^{2},$
hence%
\[
E_{2}\subseteq%
{\displaystyle\bigcup\limits_{\alpha\in I_{2}}}
B\left(  z_{\alpha}^{2},\delta^{2}\right)  \equiv E_{3}.
\]
Continuing this way, we build a family of points $\left\{  z_{\alpha}%
^{k}\right\}  _{\alpha\in I_{k}}$ for $k=1,2,3,...,$ and a family of sets
$E_{1}\subseteq E_{2}\subseteq E_{3}\subseteq...$ . We are going to show that
it is possible to choose $\delta$ small enough so that%
\begin{equation}%
{\displaystyle\bigcup_{k=1}^{\infty}}
E_{k}\subset\Omega_{n+1}. \label{Union E_k}%
\end{equation}

Namely: $E_{1}=\Omega_{n}\subset\Omega_{n+1}$ and, by definition of $E_{2}$
and (\ref{Hp 1}) $E_{2}\subset\Omega_{n+1}$ as soon as
\begin{equation}
\delta<2\varepsilon_{n}. \label{delta 1}%
\end{equation}

Let now $y\in E_{3}.$ Then there exists $z_{\alpha}^{2}\in E_{2}$ such that
$\rho\left(  y,z_{\alpha}^{2}\right)  <\delta^{2}$ and there exists $z_{\beta
}^{1}\in\Omega_{n}$ such that $\rho\left(  z_{\alpha}^{2},z_{\beta}%
^{1}\right)  <\delta.$ Since $E_{2}\subset\Omega_{n+1},$ we have $y\in
\Omega_{n+2}$ (that is $E_{3}\subset\Omega_{n+2}$) as soon as
\begin{equation}
\delta^{2}<2\varepsilon_{n+1}. \label{delta 2}%
\end{equation}
Under this assumption we can write%
\[
\rho\left(  y,z_{\beta}^{1}\right)  \leq B_{n+2}\left(  \rho\left(
y,z_{\alpha}^{2}\right)  +\rho\left(  z_{\alpha}^{2},z_{\beta}^{1}\right)
\right)  \leq B_{n+2}\left(  \delta^{2}+\delta\right)  .
\]
Then, under the further assumption%
\begin{equation}
B_{n+2}\left(  \delta^{2}+\delta\right)  \leq2\varepsilon_{n} \label{delta 3}%
\end{equation}
we can conclude that $E_{3}\subset\Omega_{n+1}$ (which strengthen the previous
conclusion $E_{3}\subset\Omega_{n+2}$).

This idea can be iterated showing that for any $y\in E_{N},$ $N=2,3,4...$
there exists $x\in\Omega_{n}$ such that%
\begin{align}
\rho\left(  x,y\right)   &  \leq B_{n+2}\left[  \delta+B_{n+2}\left[
\delta^{2}+B_{n+2}\left[  \delta^{3}+...+B_{n+2}\left[  \delta^{N-1}%
+\delta^{N}\right]  \right]  \right]  \right] \label{rho(x,y)}\\
&  \leq\sum_{k=1}^{N}\left(  B_{n+2}\delta\right)  ^{k}\leq\frac{\delta
B_{n+2}}{1-\delta B_{n+2}}<2\delta B_{n+2}\nonumber
\end{align}
provided
\begin{equation}
\delta<1/\left(  2B_{n+2}\right)  . \label{delta 5}%
\end{equation}
If, moreover,
\begin{equation}
\delta<\varepsilon_{n}/B_{n+2}, \label{delta 4}%
\end{equation}
we can conclude (\ref{Union E_k}). Choosing $\delta$ small enough to fulfill
conditions (\ref{delta 1})-(\ref{delta 4}) we are done.

Note that, for any $k=2,3...,\alpha\in I_{k},$%
\begin{equation}
z_{\alpha}^{k}\in E_{k}=%
{\displaystyle\bigcup\limits_{\beta\in I_{k-1}}}
B\left(  z_{\beta}^{k-1},\delta^{k-1}\right)  \supseteq E_{k-1} \label{E_k}%
\end{equation}
hence for any $z_{\alpha}^{k}$ there exists $\beta\in I_{k-1}$ such that
\begin{equation}
\rho\left(  z_{\alpha}^{k},z_{\beta}^{k-1}\right)  <\delta^{k-1} \label{3.8}%
\end{equation}
Moreover, for any $k=1,2,3...,$%
\begin{equation}
\rho\left(  z_{\alpha}^{k},z_{\beta}^{k}\right)  \geq\delta^{k}\text{ for any
}\alpha\neq\beta. \label{alfa not beta}%
\end{equation}

\end{proof}

After this preliminary construction we pause for a moment our proof, and give
the following

\begin{definition}
A \emph{tree} is a partial ordering $\leq$ of the set of all ordered pairs
$\left(  k,\alpha\right)  $ ($k=1,2,3,...,\alpha\in I_{k}$) which satisfies:

\begin{enumerate}
\item[(T1)] $\left(  k,\alpha\right)  \leq\left(  l,\beta\right)  \Rightarrow
k\geq l.$

\item[(T2)] For each $\left(  k,\alpha\right)  $ and $1\leq l\leq k$ there
exists a unique $\beta$ such that $\left(  k,\alpha\right)  \leq\left(
l,\beta\right)  .$

\item[(T3)] $\left(  k,\alpha\right)  \leq\left(  k-1,\beta\right)
\Longrightarrow\rho\left(  z_{\alpha}^{k},z_{\beta}^{k-1}\right)
<\delta^{k-1}.$

\item[(T4)] $\rho\left(  z_{\alpha}^{k},z_{\beta}^{k-1}\right)  <\left(
2B_{n}\right)  ^{-1}\delta^{k-1}\Longrightarrow\left(  k,\alpha\right)
\leq\left(  k-1,\beta\right)  .$
\end{enumerate}
\end{definition}

It is not restrictive to assume
\begin{equation}
B_{n}\geq2, \label{B_n}%
\end{equation}
as we will do in the following; hence the constant $\left(  2B_{n}\right)
^{-1}$ appearing in the definition is $\leq1$.

This definition is the same given in \cite{C}, except for the restriction that
our integers $k,l$ are positive. Moreover, our tree also depends on $n$
(through the points $z_{\alpha}^{k}$). As proved in \cite[Lemma 13]{C},
\emph{there exists at least one tree }(for each integer $n$). Actually, the
same proof applies in view of (\ref{3.8}), (\ref{alfa not beta}).

As in \cite{C}, we can now define the dyadic cubes:

\begin{definition}
For a fixed integer $n$, fix a tree, and let $a_{0}\in\left(  0,1\right)  $ be
a small constant to be determined. For $k=1,2,3,...,\alpha\in I_{k},$ set%
\begin{equation}
Q_{\alpha}^{k}=%
{\displaystyle\bigcup\limits_{\left(  l,\beta\right)  \leq\left(
k,\alpha\right)  }}
B\left(  z_{\beta}^{l},a_{0}\delta^{l}\right)  . \label{Def dyadic}%
\end{equation}

\end{definition}

\begin{proof}
[Proof of Theorem \ref{Main Thm}, second part]By definition, each $Q_{\alpha
}^{k}$ is an open set and (a) holds. Since
\[
B\left(  z_{\beta}^{l},a_{0}\delta^{l}\right)  \subseteq B\left(  z_{\beta
}^{l},a_{0}\delta\right)
\]
we have%
\[
Q_{\alpha}^{k}\subseteq%
{\displaystyle\bigcup\limits_{\left(  l,\beta\right)  \leq\left(
k,\alpha\right)  }}
B\left(  z_{\beta}^{l},a_{0}\delta\right)  .
\]
Since $z_{\beta}^{l}\in%
{\displaystyle\bigcup_{k=1}^{\infty}}
E_{k}\subset\Omega_{n+1},$ choosing $a_{0}$ such that%
\begin{equation}
a_{0}\delta<\varepsilon_{n+1} \label{azero 1}%
\end{equation}
we read that $Q_{\alpha}^{k}\subset\Omega_{n+2}.$ Let $y\in Q_{\alpha}^{k}$
and $z_{\beta}^{l}$ such that $\rho\left(  y,z_{\beta}^{l}\right)
<a_{0}\delta.$By (\ref{rho(x,y)}), there exists $x\in\Omega_{n}$ such that
$\rho\left(  x,z_{\beta}^{l}\right)  <2\delta B_{n+2},$ hence%
\[
\rho\left(  x,y\right)  \leq B_{n+2}\left(  \rho\left(  x,z_{\beta}%
^{l}\right)  +\rho\left(  y,z_{\beta}^{l}\right)  \right)  \leq B_{n+2}\left(
2\delta B_{n+2}+a_{0}\delta\right)
\]
so we can conclude that $y\in\Omega_{n+1}$ provided%
\begin{equation}
\delta\left(  2B_{n+2}^{2}+a_{0}B_{n+2}\right)  <2\varepsilon_{n}.
\label{delta 6}%
\end{equation}
It is now useful to choose $a_{0}=\delta;$ this implies that all the
conditions we will write on $a_{0}$ and $\delta$ simply ask that $\delta$ be
small enough in terms of the constants $\varepsilon_{n},B_{n}$, so that all
these conditions can be \emph{simultaneously }satisfied. Nevertheless, we will
keep using both the symbols $a_{0}$ and $\delta,$ to stress the different
roles of these constants.

Under assumptions (\ref{azero 1})-(\ref{delta 6}), we conclude $Q_{\alpha}%
^{k}\subset\Omega_{n+1},$ that is (b)\ holds.

From the definition (\ref{Def dyadic}) we also have the monotonicity of dyadic
cubes:%
\begin{equation}
\left(  l,\beta\right)  \leq\left(  k,\alpha\right)  \Longrightarrow Q_{\beta
}^{l}\subseteq Q_{\alpha}^{k}. \label{monotone}%
\end{equation}
By (T2) and (\ref{monotone}) we immediately have (c).

As in \cite[(3.13)]{C} we can prove that%
\begin{equation}
\left(  l,\beta\right)  \leq\left(  k,\alpha\right)  \Longrightarrow
\rho\left(  z_{\beta}^{l},z_{\alpha}^{k}\right)  \leq2B_{n+1}\delta^{k}
\label{3.13}%
\end{equation}
provided
\begin{equation}
\delta<\left(  2B_{n}\right)  ^{-1}. \label{delta 5'}%
\end{equation}
This implies (d) since, for any $x,y\in Q_{\alpha}^{k},$ $x\in B\left(
z_{\beta}^{l},a_{0}\delta^{l}\right)  ,y\in B\left(  z_{\gamma}^{h}%
,a_{0}\delta^{h}\right)  $ for some $\left(  l,\beta\right)  \leq\left(
k,\alpha\right)  ,\left(  h,\gamma\right)  \leq\left(  k,\alpha\right)  $ we
can write%
\begin{align*}
\rho\left(  x,y\right)   &  \leq B_{n+1}\left[  \rho\left(  x,z_{\beta}%
^{l}\right)  +\rho\left(  y,z_{\beta}^{l}\right)  \right] \\
&  \leq B_{n+1}\left[  a_{0}\delta^{l}+B_{n+1}\left[  \rho\left(  y,z_{\gamma
}^{h}\right)  +\rho\left(  z_{\beta}^{l},z_{\gamma}^{h}\right)  \right]
\right] \\
&  \leq B_{n+1}\left[  a_{0}\delta^{l}+B_{n+1}\left[  a_{0}\delta^{h}%
+B_{n+1}\left[  \rho\left(  z_{\beta}^{l},z_{\alpha}^{k}\right)  +\rho\left(
z_{\alpha}^{k},z_{\gamma}^{h}\right)  \right]  \right]  \right] \\
&  \leq B_{n+1}\left[  a_{0}\delta^{l}+B_{n+1}\left[  a_{0}\delta^{h}%
+B_{n+1}\left[  2B_{n+1}\delta^{k}+2B_{n+1}\delta^{k}\right]  \right]
\right]
\end{align*}
where the last inequality follows by (\ref{3.13}). Since $l,h\geq k$ this
implies (since $a_{0}<1$)%
\[
\rho\left(  x,y\right)  \leq\delta^{k}\left[  B_{n+1}+B_{n+1}^{2}+4B_{n+1}%
^{4}\right]
\]
which gives (d) with
\begin{equation}
c_{1}=7B_{n+1}^{4}. \label{c_1}%
\end{equation}
Note that with this choice of $c_{1}$ we also have $\overline{Q_{\alpha}^{k}%
}\subset B\left(  z_{\alpha}^{k},c_{1}\delta^{k}\right)  .$

In order to prove (e), we can now proceed proving, as \cite[Lemma 15]{C}:%
\begin{equation}
\text{If }Q_{\alpha}^{k}\cap Q_{\beta}^{k}\neq\emptyset\text{ then }%
\alpha=\beta. \label{disjoint}%
\end{equation}
Indeed, the same proof of \cite[Lemma 15]{C} applies, in view of (\ref{3.13}),
(b), (\ref{alfa not beta}).\ More precisely, (\ref{disjoint}) holds provided
we choose $a_{0}$ and $\delta$ small enough so that
\begin{equation}
\delta+a_{0}<\left(  2B_{n+1}\right)  ^{-3}. \label{azero 2}%
\end{equation}
With (\ref{disjoint}) in hand let us show that (e) holds. Let $l\geq k\geq1,$
$Q_{\beta}^{\ell}\cap Q_{\alpha}^{k}\neq\emptyset,$ and choose $\gamma$ such
that $\left(  l,\beta\right)  \leq\left(  k,\gamma\right)  $ (this is possible
by (T2)); then $Q_{\beta}^{l}\subseteq Q_{\gamma}^{k}$ which, together with
$Q_{\beta}^{\ell}\cap Q_{\alpha}^{k}\neq\emptyset$ implies $Q_{\alpha}^{k}\cap
Q_{\gamma}^{k}\neq\emptyset.$ By (\ref{disjoint}) then $\alpha=\gamma,$ that
is $Q_{\beta}^{l}\subseteq Q_{\alpha}^{k}$ which gives (e).

Let us come to the proof of (f). Fix $k\geq1$ and let%
\[
F_{k}=%
{\displaystyle\bigcup_{\alpha\in I_{k}}}
Q_{\alpha}^{k}.
\]
Fix $x\in\Omega_{n}=E_{1}$; since $E_{1}\subseteq E_{2}\subseteq
E_{3}\subseteq...,$ by (\ref{E_k}),
\begin{equation}
\forall h\geq1\text{ }\exists z_{\alpha}^{h}\text{ such that }\rho\left(
x,z_{\alpha}^{h}\right)  <\delta^{h}. \label{closure}%
\end{equation}
By (b), for any $h\geq k$ we have%
\[
B\left(  z_{\alpha}^{h},a_{0}\delta^{h}\right)  \subseteq Q_{\alpha}^{h}.
\]
By (c) there exists $Q_{\beta}^{k}\supseteq Q_{\alpha}^{h},$ hence
\begin{equation}
B\left(  z_{\alpha}^{h},a_{0}\delta^{h}\right)  \subseteq Q_{\alpha}%
^{h}\subseteq Q_{\beta}^{k}\subseteq F_{k}\subset\Omega_{n+1}.
\label{Lebesgue1}%
\end{equation}
By the triangle inequality,%
\begin{equation}
B\left(  z_{\alpha}^{h},a_{0}\delta^{h}\right)  \subseteq B\left(
x,B_{n+1}\left(  1+a_{0}\right)  \delta^{h}\right)  \equiv B.
\label{Lebesgue2}%
\end{equation}
In turn,
\[
B\left(  x,B_{n+1}\left(  1+a_{0}\right)  \delta^{h}\right)  \subseteq
B\left(  z_{\alpha}^{h},B_{n+1}\left(  B_{n+1}\left(  1+a_{0}\right)
\delta^{h}+\delta^{h}\right)  \right)  .
\]
For $h$ large enough we have
\begin{equation}
B_{n+1}\left(  B_{n+1}\left(  1+a_{0}\right)  \delta^{h}+\delta^{h}\right)
\leq3B_{n+1}^{2}\delta^{h}\leq\varepsilon_{n}, \label{delta 7}%
\end{equation}
and the local doubling condition (\ref{Hp 3}) implies%
\[
\mu\left(  B\left(  z_{\alpha}^{h},a_{0}\delta^{h}\right)  \right)  \geq
c\mu\left(  B\right)
\]
for some constant $c>0$ depending on $n$ (once we have fixed $\delta$ and
$a_{0}$). By (\ref{Lebesgue1}) and (\ref{Lebesgue2}) the last inequality gives%
\[
\frac{\mu\left(  F_{k}\cap B\right)  }{\mu\left(  B\right)  }\geq\frac
{\mu\left(  B\left(  z_{\alpha}^{h},a_{0}\delta^{h}\right)  \cap B\right)
}{\mu\left(  B\right)  }=\frac{\mu\left(  B\left(  z_{\alpha}^{h},a_{0}%
\delta^{h}\right)  \right)  }{\mu\left(  B\right)  }\geq c>0
\]
and $h$ large enough. Letting $h\rightarrow+\infty$ we find that%
\[
\underset{r\rightarrow0}{\text{limsup}}\frac{\mu\left(  F_{k}\cap B\left(
x,r\right)  \right)  }{\mu\left(  B\left(  x,r\right)  \right)  }\geq
c>0\text{ \ }\forall x\in\Omega_{n},k=1,2,3...
\]
By Lebesgue's theorem on differentiation of the integral, $\mu\left(
\Omega_{n}\setminus F_{k}\right)  =0.$ Letting
\begin{equation}
E=%
{\displaystyle\bigcup_{k=1}^{\infty}}
\left(  \Omega_{n}\setminus F_{k}\right)  \label{E}%
\end{equation}
we have (f).

To prove (g) we need a refinement of the argument used in the above proof of
(f). Since $\Omega_{n}=E_{1}\subseteq E_{2}\subseteq E_{3}\subseteq...,$ by
(\ref{E_k}) for any $x\in E_{N}$ we have that:
\[
\forall h\geq N\text{ }\exists z_{\alpha}^{h}\text{ such that }\rho\left(
x,z_{\alpha}^{h}\right)  <\delta^{h}%
\]
while for any $h\geq k$ (\ref{Lebesgue1}) and (\ref{Lebesgue2}) still hold.
Hence we can prove as above that%
\begin{equation}
\mu\left(  E_{h}\setminus F_{k}\right)  =0\text{ for any }k,h\geq1.
\label{refine}%
\end{equation}

Let $F$ be the null set given by $%
{\displaystyle\bigcup_{h,k\geq1}}
\left(  E_{h}\setminus F_{k}\right)  .$ Then fix a dyadic cube $Q_{\alpha}%
^{k}$ and pick a point $x\in Q_{\alpha}^{k}\setminus F$. Since $x\in
Q_{\alpha}^{k},$ there exists $B\left(  z_{\beta}^{h},a_{0}\delta^{h}\right)
\ni x$ for some $h\geq k;$ since
\[
B\left(  z_{\beta}^{h},a_{0}\delta^{h}\right)  \subset B\left(  z_{\beta}%
^{h},\delta^{h}\right)  \subset E_{h},
\]
this means that $x\in E_{h};$ since $x\notin F$, (\ref{refine}) implies that
for any $l\geq1$ the point $x$ belongs to some $Q_{\beta}^{l},$ which is (g).
Clearly, the fact that the null set $F$ appearing in the proof of this point
is possibly different from the null set $E$ appearing in the proof of point
(f) is immaterial, since we can always relabel $E$ the union of the two.

To prove (h), let $x\in Q_{\alpha}^{k}\setminus F$ (with $F$ as above) and
$r>0.$ We need to establish a lower bound on $\mu\left(  B\left(  x,r\right)
\cap Q_{\alpha}^{k}\right)  ;$ let us distinguish two cases:

(i) $r\leq\delta^{k}.$ Let $j\geq k$ such that $\delta^{j+1}<r\leq\delta^{j}$
and let $Q_{\beta}^{j+2}$ a cube containing $x$ (by (g) it certainly exists).
By (e), $Q_{\beta}^{j+2}\subset Q_{\alpha}^{k}$ while by (d), diam$\left(
Q_{\beta}^{j+2}\right)  \leq c_{1}\delta^{j+2}.$ Then $Q_{\beta}^{j+2}\subset
B\left(  x,r\right)  ,$ since, for $y\in Q_{\beta}^{j+2},$%
\[
\rho\left(  x,y\right)  \leq B_{n+1}\left[  \rho\left(  x,z_{\beta}%
^{j+2}\right)  +\rho\left(  y,z_{\beta}^{j+2}\right)  \right]  \leq
2B_{n+1}c_{1}\delta^{j+2}\leq\delta^{j+1}<r
\]
provided $\delta$ is small enough so that
\begin{equation}
2B_{n+1}c_{1}\delta\leq1. \label{delta 9}%
\end{equation}
Therefore%
\begin{align*}
\mu\left(  B\left(  x,r\right)  \cap Q_{\alpha}^{k}\right)   &  \geq\mu\left(
Q_{\beta}^{j+2}\right)  \geq\mu\left(  B\left(  z_{\beta}^{j+2},a_{0}%
\delta^{j+2}\right)  \right) \\
&  \geq c_{0}\mu\left(  B\left(  x,\delta^{j}\right)  \right)  \geq c_{0}%
\mu\left(  B\left(  x,r\right)  \right)
\end{align*}
where the up to last inequality follows by the local doubling condition
(\ref{Hp 3}), with a constant $c_{0}$ only depending on $n.$

(ii) $r>\delta^{k}.$ Let $Q_{\beta}^{k+1}\ni x$ (by point (g) it certainly
exists). Since diam$\left(  Q_{\beta}^{k+1}\right)  \leq c_{1}\delta^{k+1},$
\[
Q_{\beta}^{k+1}\subset B\left(  x,c_{1}\delta^{k+1}\right)  \subset B\left(
x,r\right)
\]
as soon as%
\begin{equation}
c_{1}\delta<1 \label{delta 8}%
\end{equation}
hence, by point (a),
\[
\mu\left(  B\left(  x,r\right)  \cap Q_{\alpha}^{k}\right)  \geq\mu\left(
Q_{\beta}^{k+1}\right)  \geq\mu\left(  B\left(  z_{\beta}^{k+1},a_{0}%
\delta^{k+1}\right)  \right)
\]
while, since $z_{\beta}^{k+1}\in Q_{\alpha}^{k}$ and diam$\left(  Q_{\alpha
}^{k}\right)  \leq c_{1}\delta^{k},$%
\[
\mu\left(  B\left(  z_{\beta}^{k+1},c_{1}\delta^{k}\right)  \right)  \geq
\mu\left(  Q_{\alpha}^{k}\right)  .
\]
To conclude (\ref{lower bounds cubes}), which immediately give (h), we have to
apply the local doubling condition, to say that
\[
\mu\left(  B\left(  z_{\beta}^{k+1},c_{1}\delta^{k}\right)  \right)  \leq
c_{0}\mu\left(  B\left(  z_{\beta}^{k+1},a_{0}\delta^{k+1}\right)  \right)  .
\]
This is possible, once we have (at last) fixed $\delta,$ with some constant
depending on $\delta,$ and therefore on $n$. Hence Theorem \ref{Main Thm} is proved.

Finally, note that in our iterative construction, at every step $n$ we can
always choose the number $\delta_{\left(  n\right)  }$ less than or equal to
the number $\delta_{\left(  n-1\right)  }$ chosen at the previous step. Hence
the sequence $\delta_{\left(  n\right)  }$ can be assumed to be nonincreasing.
\end{proof}

\begin{remark}
In the previous proof the reader could be confused by the number of conditions
we have imposed on $\delta$ and the other constants. So, let us summarize the
logical line of this procedure. First, we can assume without loss of
generality that the parameters $\varepsilon_{n},B_{n}$ of $\Omega_{n}$ satisfy
the following:%
\begin{align*}
B_{n+1}  &  \geq B_{n}\geq2\text{ for every }n\text{;}\\
\varepsilon_{n+1}  &  \leq\varepsilon_{n}\leq\frac{1}{2}\text{ for every
}n\text{.}%
\end{align*}
Then we have chosen%
\[
c_{1}=7B_{n+1}^{4}%
\]
and $a_{0}=\delta$, where $\delta$ has to satisfy conditions (\ref{delta 1}),
(\ref{delta 2}), (\ref{delta 3}), (\ref{delta 5}), (\ref{delta 4}),
(\ref{azero 1}), (\ref{delta 6}), (\ref{delta 5'}), (\ref{azero 2}),
(\ref{delta 9}), (\ref{delta 8}), and also (\ref{delta 10}), which will be
used in the proof of Lemma \ref{Lemma pos meas}. With some patience one can
check that a possible choice is
\[
\delta=\frac{1}{2}\min\left(  \varepsilon_{n+1},\frac{\varepsilon_{n}%
}{4B_{n+2}^{2}},\frac{1}{14B_{n+1}^{5}}\right)  .
\]
After $\delta$ has been fixed, the constants $c_{0},c_{2}$ can be determined
in terms of $\delta$ and $C_{n}$.
\end{remark}

Point (h) of the above theorem means that \emph{each set }$Q_{\alpha}%
^{k}\setminus E$ \emph{is a space of homogeneous type}. It is useful to
reinforce the previous statement with the following:

\begin{proposition}
\label{Prop Q homogeneous}For each $\Omega_{n},k$ and $\alpha\in I_{k},$ the
set $Q_{\alpha}^{k}$ is a space of homogeneous type.
\end{proposition}

Here and in the following, whenever we will write that a set $S\subset\Omega$
is a space of homogeneous type we will mean that $\left(  S,\rho,d\mu\right)
$ is a space of homogeneous type, \emph{with respect to the same }$\rho$
\emph{and} $\mu$ \emph{already defined in} $\Omega$.

\begin{proof}
The only point to prove is that if $x\in Q_{\alpha}^{k}\cap E$ (where $E$ is
like in Theorem \ref{Main Thm}) then%
\[
\mu\left(  B\left(  x,2r\right)  \cap Q_{\alpha}^{k}\right)  \leq c\mu\left(
B\left(  x,r\right)  \cap Q_{\alpha}^{k}\right)  \text{ for any }r>0.
\]
Pick $y\in B\left(  x,\varepsilon r\right)  \cap\left(  Q_{\alpha}%
^{k}\setminus E\right)  $ for some small $\varepsilon$ to be fixed later. Such
$y$ certainly exists, otherwise $E$ would contain the open set $B\left(
x,\varepsilon r\right)  \cap Q_{\alpha}^{k},$ which by (\ref{Hp 3}) has
positive measure, while $E$ has zero measure.

Since $x\in Q_{\alpha}^{k}\subset\Omega_{n+1},$ for $r\leq\varepsilon_{n+1}$
we have $B\left(  x,2r\right)  \subset\Omega_{n+2}$ and we can prove
\begin{equation}
B\left(  x,2r\right)  \subset B\left(  y,c_{1}r\right)  \label{hom inclu1}%
\end{equation}
with $c_{1}=\left(  2+\varepsilon\right)  B_{n+2}.$ Analogously,%
\begin{equation}
B\left(  y,c_{2}r\right)  \subset B\left(  x,r\right)  \label{hom inclu2}%
\end{equation}
provided $\left(  \varepsilon+c_{2}\right)  B_{n+2}<1.$ Hence
(\ref{hom inclu1}), (\ref{hom inclu2}) hold for suitable constants
$c_{2}<1<c_{1}$ and $\varepsilon$ small enough (depending on $n$ but not on
$r$), while by point (h) of Theorem \ref{Main Thm}, since $y\in$ $Q_{\alpha
}^{k}\setminus E$ we have%
\[
\mu\left(  B\left(  y,c_{1}r\right)  \cap Q_{\alpha}^{k}\right)  \leq
c\mu\left(  B\left(  y,c_{2}r\right)  \cap Q_{\alpha}^{k}\right)
\]
for some constant $c$ depending on $c_{1},c_{2}$ and any $r>0$. So we conclude%
\[
\mu\left(  B\left(  x,2r\right)  \cap Q_{\alpha}^{k}\right)  \leq c\mu\left(
B\left(  x,r\right)  \cap Q_{\alpha}^{k}\right)  \text{ for any }%
r\leq\varepsilon_{n+1}.
\]

Let now $r>\varepsilon_{n+1}$. Pick $y\in B\left(  x,\varepsilon
_{n+1}/2B_{n+2}\right)  \cap\left(  Q_{\alpha}^{k}\setminus E\right)  $. By
(g), for any $h$ there exists $Q_{\beta}^{h}\ni y.$ Since diam$\left(
Q_{\beta}^{h}\right)  \leq c_{1}\delta^{h},$ for $z\in$ $Q_{\beta}^{h}$ we
have%
\[
\rho\left(  z,x\right)  \leq B_{n+2}\left(  \rho\left(  z,y\right)
+\rho\left(  y,x\right)  \right)  \leq B_{n+2}c_{1}\delta^{h}+\frac
{\varepsilon_{n+1}}{2}<\varepsilon_{n+1}%
\]
for $h$ large enough. Let $h_{0}$ the minimum integer $\geq k$ such that this
is true, so that $Q_{\beta}^{h_{0}}\subset B\left(  x,\varepsilon
_{n+1}\right)  $. Hence%
\[
\mu\left(  B\left(  x,r\right)  \cap Q_{\alpha}^{k}\right)  \geq\mu\left(
Q_{\beta}^{h_{0}}\right)
\]
while $\mu\left(  B\left(  x,2r\right)  \cap Q_{\alpha}^{k}\right)  \leq
\mu\left(  Q_{\alpha}^{k}\right)  .$ The desired conclusion follows since
$\mu\left(  Q_{\alpha}^{k}\right)  $ and $\mu\left(  Q_{\beta}^{h_{0}}\right)
$ are comparable. (See the last part of the proof of Theorem \ref{Main Thm}).
So we are done.
\end{proof}

By Theorem \ref{Main Thm}, point (f), we know that each family of cubes
$\left\{  Q_{\alpha}^{k}\right\}  _{\alpha\in I_{k}}$ covers $\Omega
_{n}\setminus E.$ Since the cubes $Q_{\alpha}^{k}$ are open and disjoint sets,
it is reasonable that they cannot generally cover the whole $\Omega_{n}$ (if,
for instance, $\Omega_{n}$ is a connected set, this is impossible). On the
other hand, from the proof of the theorem we can read the following fact:

\begin{proposition}
\label{prop closure}For any $\Omega_{n}$ and any $k=1,2,3,...,$ the closure of
$%
{\displaystyle\bigcup_{\alpha\in I_{k}}}
Q_{\alpha}^{k}$ covers $\Omega_{n}.$
\end{proposition}

\begin{proof}
Let $x\in\Omega_{n}.$ By (\ref{closure}), $\forall h\geq1$ $\exists
z_{\alpha_{h}}^{h}$ such that $\rho\left(  x,z_{\alpha_{h}}^{h}\right)
<\delta^{h}.$ Hence the sequence $\left\{  z_{\alpha_{k}}^{k}\right\}
_{k=1}^{\infty}$ converges to $x$. Moreover, the point $z_{\alpha_{1}}^{1}$
belongs to $Q_{\alpha_{1}}^{1}\subset\cup_{\beta\in I_{1}}Q_{\beta}^{1};$ the
point $z_{\alpha_{2}}^{2}$ belongs to a cube $Q_{\alpha_{2}}^{2}$ which is
contained in some parent cube $Q_{\gamma}^{1}\subset\cup_{\beta\in I_{1}%
}Q_{\beta}^{1},$ and so on. Hence the whole sequence is contained in
$\cup_{\beta\in I_{1}}Q_{\beta}^{1},$ which means that $x$ belongs to the
closure of $\cup_{\beta\in I_{1}}Q_{\beta}^{1}.$ With the same reasoning we
can say that for any positive integer $h$ the sequence $\left\{  z_{\alpha
_{k}}^{k}\right\}  _{k=h}^{\infty}$ is contained in $\cup_{\beta\in I_{h}%
}Q_{\beta}^{h},$ hence $x$ belongs to the closure of $\cup_{\beta\in I_{h}%
}Q_{\beta}^{h}$ for any $h=1,2,3...$
\end{proof}

The next question we pose is: how many cubes form each family $\left\{
Q_{\alpha}^{k}\right\}  _{\alpha\in I_{k}}$? We expect them to be finitely
many, since they are contained in $\Omega_{n+1},$ which is bounded, they are
pairwise disjoint and have essentially the same diameter. A formal proof of
this fact requires some care. We first need the following lemma, which will be
useful also other times.

\begin{lemma}
\label{Lemma pos meas}For any $k=1,2,3,...$ there exists $c_{n,k}>0$ such that%
\[
\inf_{z\in\Omega_{n}}\mu\left(  B\left(  z,a_{0}\delta^{k}\right)  \right)
\geq c_{n,k}%
\]
where $\delta$ and $a_{0}$ are as in Theorem \ref{Main Thm}.
\end{lemma}

\begin{proof}
Since $\overline{\Omega}_{n}$ is compact (see assumption (H4)), there exists a
finite number of points $z_{1},...,z_{N}\in\overline{\Omega}_{n}$ such that
\[
\overline{\Omega}_{n}\subset%
{\displaystyle\bigcup\limits_{i=1}^{N}}
B\left(  z_{i},a_{0}\delta^{k}\right)  .
\]
Let now $z$ be any point of $\Omega_{n};$ there exists $i_{0}$ such that $z\in
B\left(  z_{i_{0}},a_{0}\delta^{k}\right)  .$ On the other hand, for any pair
of nondisjoint balls of radius $r$ and centers $z,z_{i_{0}},$ we have the
inclusion $B\left(  z_{i_{0}},r\right)  \subset B\left(  z,B_{n+1}\left(
2B_{n+1}+1\right)  r\right)  .$ Assuming
\begin{equation}
\left(  2B_{n+1}+1\right)  a_{0}\delta\leq2\varepsilon_{n} \label{delta 10}%
\end{equation}
we have, by the doubling condition (\ref{Hp 3})%
\[
\mu\left(  B\left(  z,a_{0}\delta^{k}\right)  \right)  \geq c\mu\left(
B\left(  z,\left(  2B_{n}+1\right)  a_{0}\delta^{k}\right)  \right)  \geq
c\mu\left(  B\left(  z_{i_{0}},a_{0}\delta^{k}\right)  \right)  \geq
c\varepsilon\equiv c_{n,k}%
\]
having set%
\[
\varepsilon=\min_{i=1,2,...,N}\mu\left(  B\left(  z_{i},a_{0}\delta
^{k}\right)  \right)  .
\]

\end{proof}

\begin{proposition}
For each $k=1,2,3,...,$ the family $\left\{  Q_{\alpha}^{k}\right\}
_{\alpha\in I_{k}}$ is finite.
\end{proposition}

\begin{proof}
Since $%
{\displaystyle\bigcup\limits_{\alpha\in I_{k}}}
Q_{\alpha}^{k}\subset\Omega_{n+1},$ we have%
\[
\mu\left(
{\displaystyle\bigcup\limits_{\alpha\in I_{k}}}
Q_{\alpha}^{k}\right)  \leq\mu\left(  \Omega_{n+1}\right)  <\infty,
\]
(recall that any $\Omega_{n+1}$ has finite measure, as noted after definition
\ref{Def loc hom space}). Since the $Q_{\alpha}^{k}$'s are pairwise disjoint
and by Theorem \ref{Main Thm}, point (a), $Q_{\alpha}^{k}\supset B\left(
z_{\alpha}^{k},a_{0}\delta^{k}\right)  ,$%
\[
\mu\left(
{\displaystyle\bigcup\limits_{\alpha\in I_{k}}}
Q_{\alpha}^{k}\right)  =\sum_{\alpha\in I_{k}}\mu\left(  Q_{\alpha}%
^{k}\right)  \geq\sum_{\alpha\in I_{k}}\mu\left(  B\left(  z_{\alpha}%
^{k},a_{0}\delta^{k}\right)  \right)
\]
where the last sum, by the previous Lemma, is an infinite quantity unless
$I_{k}$ is finite. Therefore $I_{k}$ is finite.
\end{proof}

The finiteness of the covering $\left\{  Q_{\alpha}^{k}\right\}  _{\alpha\in
I_{k}}$ of $\Omega_{n}$ at any scale $k$ is interesting for the following consequence:

\begin{corollary}
\label{corollary covering homog space}For any $k=1,2,3,...$ the set
\[
F_{k}=%
{\displaystyle\bigcup\limits_{\alpha\in I_{k}}}
Q_{\alpha}^{k}%
\]
is a space of homogeneous type. The same conclusion holds for the union of
\emph{any }subfamily of $\left\{  Q_{\alpha}^{k}\right\}  _{\alpha\in I_{k}}$.
The doubling constants depends on $n$ and $k.$
\end{corollary}

\begin{proof}
We have to prove that%
\begin{equation}
\mu\left(  B\left(  x,2r\right)  \cap F_{k}\right)  \leq c\mu\left(  B\left(
x,r\right)  \cap F_{k}\right)  \text{ for any }r>0,x\in F_{k}.
\label{C^k homogeneous}%
\end{equation}
Let us first prove this inequality when $x\in F_{k}\setminus E,$ where $E$ is
the null set appearing in Theorem \ref{Main Thm}. So, let $x\in Q_{\alpha}%
^{k}\setminus E$ for some $\alpha\in I_{k}$ and let $r>0.$ We will apply
(\ref{lower bounds cubes}) in Theorem \ref{Main Thm}, distinguishing the cases
$r\leq\delta^{k}$ and $r>\delta^{k}.$

When $r\leq\delta^{k}$, by the doubling condition (\ref{Hp 3}) we have
\begin{align*}
\mu\left(  B\left(  x,r\right)  \cap F_{k}\right)   &  \geq\mu\left(  B\left(
x,r\right)  \cap Q_{\alpha}^{k}\right)  \geq c_{0}\mu\left(  B\left(
x,r\right)  \right) \\
&  \geq\frac{c_{0}}{C_{n}}\mu\left(  B\left(  x,2r\right)  \right)  \geq
\frac{c_{0}}{C_{n}}\mu\left(  B\left(  x,2r\right)  \cap F_{k}\right)  .
\end{align*}
When $r>\delta^{k}$%
\begin{align*}
\mu\left(  B\left(  x,r\right)  \cap F_{k}\right)   &  \geq\mu\left(  B\left(
x,r\right)  \cap Q_{\alpha}^{k}\right)  \geq c_{0}\mu\left(  Q_{\alpha}%
^{k}\right) \\
&  \geq c_{n,k}\mu\left(  F_{k}\right)  \geq c_{n,k}\mu\left(  B\left(
x,2r\right)  \cap F_{k}\right)
\end{align*}
where in the up to last inequality we have used the fact that the $Q_{\alpha
}^{k}$ are finitely many open sets, each of positive measure, while
$\mu\left(  F_{k}\right)  \leq\mu\left(  \Omega_{n+1}\right)  <\infty,$ so
that for some constant $c$ depending on $n$ and $k$ (but not on $\alpha$), we
can write $\mu\left(  Q_{\alpha}^{k}\right)  \geq c\mu\left(  F_{k}\right)  .$

If now $x\in F_{k}\cap E,$ we can repeat the same reasoning used in the proof
of Proposition \ref{Prop Q homogeneous} to show that (\ref{C^k homogeneous})
still holds. This completes the proof.
\end{proof}

Summarizing several results proved so far, we can say that for any $n$ there
exists a space of homogeneous type $F_{k},$ contained in $\Omega_{n+1}$ and
essentially containing $\Omega_{n},$ in the sense that $\Omega_{n}\setminus
E\subset F_{k}$ (by Theorem \ref{Main Thm}, f) and $\Omega_{n}\subset
\overline{F}_{k}$ (by Proposition \ref{prop closure}). In view of our
applications to singular integrals, it is important to get a local and more
quantitative version of this result. This is contained in the following
theorem, which is the main result in this section. Since it involves different
sets $\Omega_{n},$ here we have to add an index $n$ to the number $\delta$ and
the cubes $Q_{\alpha}^{k}$.

\begin{theorem}
\label{Thm F}For every $n$ there exists $R_{n}>0$ such that for any
$\overline{x}\in\Omega_{n}$ and $R\leq R_{n}$ there exists an open set $F$
such that:

(i) $F$ is a space of homogeneous type; its doubling constant depends on $n$
but not on $R$;

(ii) $B\left(  \overline{x},R\right)  \setminus E\subset F\subset\Omega_{n+2}$
(with $\mu\left(  E\right)  =0$);

(iii) $B\left(  \overline{x},R\right)  \subset\overline{F};$

(iv) diam$F\leq cR$ for some constant $c$ depending on $n$ but not on $R$;

(v) $\mu\left(  F\right)  \leq c\mu\left(  B\left(  \overline{x},R\right)
\right)  $ for some constant $c$ depending on $n$ but not on $R$.
\end{theorem}

The independence of the constants from $R$ will be precious when dealing with
commutators of singular or fractional integrals with $VMO$ functions. Clearly,
the whole $\Omega_{n}$ can be covered, for any $R\leq R_{n}$, by a finite
number of balls $B\left(  x_{i},R\right)  ,$ to which this theorem is applicable.

\begin{remark}
The reader could ask why we do not consider the set $\overline{F}$ (which
satisfies the simple inclusions $B\left(  \overline{x},R\right)
\subset\overline{F}\subset\Omega_{n+2}$) instead of $F$ (which does not
exactly contain $B\left(  \overline{x},R\right)  $). The problem with
$\overline{F}$ is that, in our abstract context, it is not obvious how to
prove that it is a space of homogeneous type, too.
\end{remark}

\begin{proof}
Fix $\overline{x}\in\Omega_{n}$ and let $R_{n}=\delta_{\left(  n\right)
}^{k_{0}}$ for a $k_{0}$ to be chosen later, but such that $R_{n}%
\leq2\varepsilon_{n}$, hence $B\left(  \overline{x},R\right)  \subset
\Omega_{n+1}.$ For $R\leq R_{n},$ pick $k\geq k_{0}$ such that $\delta
_{\left(  n\right)  }^{k+1}<R\leq\delta_{\left(  n\right)  }^{k}.$ For these
$n$ and $k,$ there exists $\alpha\in I_{k}^{\left(  n\right)  }$ such that
(see Theorem \ref{Main Thm})%
\[
\overline{x}\in Q_{\alpha}^{\left(  n\right)  ,k}\subset B\left(  z_{\alpha
}^{\left(  n\right)  ,k},c_{1,\left(  n\right)  }\delta_{\left(  n\right)
}^{k}\right)  \subset\Omega_{n+1}.
\]
For any $y\in B\left(  \overline{x},R\right)  $ we can write%
\[
\rho\left(  y,z_{\alpha}^{\left(  n\right)  ,k}\right)  \leq B_{n+1}\left(
\rho\left(  y,\overline{x}\right)  +\rho\left(  \overline{x},z_{\alpha
}^{\left(  n\right)  ,k}\right)  \right)  <B_{n+1}\left(  \delta_{\left(
n\right)  }^{k}+c_{1,\left(  n\right)  }\delta_{\left(  n\right)  }%
^{k}\right)  \equiv h_{n}\delta_{\left(  n\right)  }^{k}%
\]
hence%
\[
B\left(  \overline{x},R\right)  \subset B\left(  z_{\alpha}^{\left(  n\right)
,k},h_{n}\delta_{\left(  n\right)  }^{k}\right)
\]
with%
\[
h_{n}=B_{n+1}\left(  1+c_{1,\left(  n\right)  }\right)  .
\]
Choose $k_{0}$ (and consequently $R_{n}$) so that $h_{n}\delta_{\left(
n\right)  }^{k_{0}}\leq2\varepsilon_{n},$ hence
\[
B\left(  z_{\alpha}^{\left(  n\right)  ,k},h_{n}\delta_{\left(  n\right)
}^{k}\right)  \subset\Omega_{n+1}%
\]
for any $k\geq k_{0},$ and so for any $R\leq R_{n}$. Since $\Omega_{n+1}$ is
covered (up to a null set) by the union of all the dyadic cubes $Q_{\beta
}^{\left(  n+1\right)  ,k},$ we can define the set%
\[
F=%
{\displaystyle\bigcup}
\left\{  Q_{\beta}^{\left(  n+1\right)  ,k}:Q_{\beta}^{\left(  n+1\right)
,k}\cap B\left(  z_{\alpha}^{\left(  n\right)  ,k},h_{n}\delta_{\left(
n\right)  }^{k}\right)  \neq\emptyset\right\}
\]
and we immediately get%
\[
B\left(  \overline{x},R\right)  \setminus E\subset B\left(  z_{\alpha
}^{\left(  n\right)  ,k},h_{n}\delta_{\left(  n\right)  }^{k}\right)
\setminus E\subset F\subset\Omega_{n+2},
\]
that is (ii). Moreover, by Proposition \ref{prop closure} we also have
$B\left(  \overline{x},R\right)  \subset\overline{F},$ which is (iii).

By Corollary \ref{corollary covering homog space}, $F\ $is a space of
homogeneous type. Note that, for the moment, we only know that its doubling
constant depends on $n$ and $k$ (that is on $R$); we want to prove that it
actually only depends on $n$.

Since%
\[
\text{diam}Q_{\beta}^{\left(  n+1\right)  ,k}<c_{1,\left(  n+1\right)  }%
\delta_{\left(  n+1\right)  }^{k}\leq c_{1,\left(  n+1\right)  }%
\delta_{\left(  n\right)  }^{k}%
\]
(the sequence $\delta_{\left(  n\right)  }$ is nonincreasing) and each of the
cubes defining $F$ intersects $B\left(  z_{\alpha}^{\left(  n\right)
,k},h_{n}\delta_{\left(  n\right)  }^{k}\right)  ,$ the quasitriangle
inequality in $\Omega_{n+2}$ gives diam$F\leq c\delta_{\left(  n\right)  }%
^{k}$ for some constant $c$ depending on $n,$ that is (iv). Finally, since
$B\left(  \overline{x},R\right)  \supset B\left(  \overline{x},\delta_{\left(
n\right)  }^{k+1}\right)  $, repeated applications of the quasitriangle
inequality in $\Omega_{n+2}$ give $F\subset B\left(  \overline{x}%
,j_{n}R\right)  $ for some constant $j_{n}$ dependent on $n$ but not on $R$.
Shrinking if necessary the number $R_{n}$ (that is enlarging the integer
$k_{0}$) we can assure that the local doubling condition in $\Omega_{n+2}$ is
applicable to the ball $B\left(  \overline{x},j_{n}R\right)  $ for $R\leq
R_{n}$ and conclude that%
\[
\mu\left(  F\right)  \leq\mu\left(  B\left(  \overline{x},j_{n}R\right)
\right)  \leq cB\left(  \overline{x},R\right)
\]
for some constant $c$ depending on $n$ but not on $R$, that is (v). This also
implies that $\mu\left(  F\right)  $ is comparable to $\mu\left(  Q_{\beta
}^{\left(  n+1\right)  ,k}\right)  $ for any of the cubes defining $F$. Hence
we can now prove that the doubling constant of $F$ only depends on $n$.
Namely, revising the last part of the proof of Corollary
\ref{corollary covering homog space} we can see that inequality%
\[
c_{0}\mu\left(  Q_{\alpha}^{k}\right)  \geq c_{n,k}\mu\left(  F_{k}\right)
\]
now rewrites as
\[
c_{0}\mu\left(  Q_{\beta}^{\left(  n+1\right)  ,k}\right)  \geq c_{n}%
\mu\left(  F\right)
\]
and we are done.
\end{proof}

\section{H\"{o}lder continuous functions\label{Section Holder}}

In several problems related to singular or fractional integrals we will need
H\"{o}lder continuous cutoff functions adapted to concentric balls. This
construction is classical and does not depend on the doubling condition, so
can be performed in any $\Omega_{n}$ as in usual spaces of homogeneous type.

Fix $\Omega_{n}.$ The function $\rho$ is a quasidistance in $\Omega_{n},$
hence by known results of Macias-Segovia \cite[Thm. 2]{MS} we can build a new
quasidistance $d$ in $\Omega_{n},$ equivalent to $\rho$ in $\Omega_{n}$, and
such that for some $\alpha\in\left(  0,1\right)  $ $d$ is \emph{of order
}$\alpha,$ which means that%
\begin{equation}
\left\vert d\left(  x_{1},y\right)  -d\left(  x_{2},y\right)  \right\vert \leq
cd\left(  x_{1},x_{2}\right)  ^{\alpha}\left\{  d\left(  x_{1},y\right)
^{1-\alpha}+d\left(  x_{2},y\right)  ^{1-\alpha}\right\}  \label{MS}%
\end{equation}
for some constant $c>0,$ any $x,y,z\in\Omega_{n}.$ Here and in the following,
saying that two functions $\rho_{1}\left(  x,y\right)  ,\rho_{2}\left(
x,y\right)  $ are equivalent in $\Omega_{n}$ means that for two positive
constants $c_{1},c_{2}>0$ we have%
\[
c_{1}\rho_{1}\left(  x,y\right)  \leq\rho_{2}\left(  x,y\right)  \leq
c_{2}\rho_{1}\left(  x,y\right)  \text{ for any }x,y\in\Omega_{n}.
\]

It is worthwhile to note that the exponent $\alpha$ in (\ref{MS}) depends on
$n;$ from the proof given in \cite[Thm. 2]{MS} we read $\alpha=1/\log
_{2}\left(  3B_{n}^{2}\right)  ,$ which is not optimal in the sense that for
$B_{n}=1$ (that is when $\rho$ is a distance) does not say that (\ref{MS})
holds with $\alpha=1$.

Let us write $B_{d}\left(  x,r\right)  $ for the $d$-ball of center $x$ and
radius $r.$ Now, for any $x_{0}\in\Omega_{n}$ with $B_{d}\left(
x_{0},2r\right)  \subset\Omega_{n}$ we can define the function%
\[
\phi\left(  x\right)  =\psi\left(  d\left(  x,x_{0}\right)  \right)
\]
where
\[
\psi\left(  t\right)  =\left\{
\begin{array}
[c]{ll}%
1 & 0\leq t\leq r\\
2-t/r & r\leq t\leq2r\\
0 & t\geq2r
\end{array}
\right.  .
\]
A standard computation exploiting (\ref{MS}) and the equivalence between
$\rho$ and $d$ allows to prove the following:

\begin{proposition}
\label{Prop cutoff}For any $\Omega_{n}$ there exists an exponent $\alpha>0$
and two constants $c_{1}<1,c_{2}>2,$ such that for any $x_{0}\in\Omega_{n}$
and $r>0$ with $B\left(  x,c_{2}r\right)  \subset\Omega_{n}$ there exists a
function $\phi$ with the following properties:%
\begin{align*}
0  &  \leq\phi\left(  x\right)  \leq1;\\
\phi\left(  x\right)   &  =1\text{ for }x\in B\left(  x_{0},c_{1}r\right) \\
\phi\left(  x\right)   &  =0\text{ for }x\notin B\left(  x_{0},c_{2}r\right)
\\
\left\vert \phi\left(  x_{1}\right)  -\phi\left(  x_{2}\right)  \right\vert
&  \leq c\left(  \frac{\rho\left(  x_{1},x_{2}\right)  }{r}\right)  ^{\alpha
}\text{ for any }x_{1},x_{2}\in\Omega_{n}.
\end{align*}

\end{proposition}

The cutoff function $\phi$ belongs to the space $C_{0}^{\alpha}\left(
\Omega_{n}\right)  .$ Note that we can build such cutoff functions only for
$\alpha\leq\alpha_{0}$ where the threshold $\alpha_{0}$ depends on the space
$\Omega_{n}.$ We will briefly write%
\[
\phi\in C_{0}^{\alpha}\left(  \Omega_{n}\right)  ,\text{ }B\left(  x_{0}%
,c_{1}r\right)  \prec\phi\prec B\left(  x_{0},c_{2}r\right)
\]
to say that $\phi$ has all the properties stated in the above proposition.

By our assumption of regularity of the measure $\mu$, the above result
\cite[Thm. 2]{MS} also implies, by a fairly standard argument, that for any
bounded Borel set $E$ we can build a H\"{o}lder continuous function which
approximates in $L^{p}$ norm, for any $p\in\lbrack1,\infty)$, the
characteristic function of $E$. Therefore the following density result holds:

\begin{proposition}
\label{Prop density}For any $\Omega_{n}$ there exists $\alpha_{0}>0,$
depending on $n$, such that for any $\alpha\in(0,\alpha_{0}],$ any
$p\in\lbrack1,\infty),$ the space $C^{\alpha}\left(  \Omega_{n}\right)  $ is
dense in $L^{p}\left(  \Omega_{n}\right)  $. If $\rho$ is a distance we can
take $\alpha_{0}=1.$
\end{proposition}

We leave the details to the interested reader. Note that this is the only
point of the theory where we use the assumption of regularity of $\mu$;
moreover, this assumption could actually be removed. Namely, refining an
argument contained in \cite[Thm. 2.2.2]{F}, it is possible to prove that,
under our assumptions (H1), (H2), (H4), any positive Borel measure on $\Omega$
has the regularity property which is used in the proof of this proposition.

\section{Local singular and fractional integrals\label{section singular}}

We now want to develop a theory of singular and fractional integrals in
locally homogeneous space. We are interested in situations, which typically
occur when dealing with local a-priori estimates for subelliptic PDEs, where
one builds singular kernels $K\left(  x,y\right)  $ which are naturally
defined only locally, say for $x,y$ belonging to some ball $B\left(
\overline{x},R_{0}\right)  \subset\Omega_{n+1}$ with $\overline{x}\in
\Omega_{n}.$ Starting from this kernel, one builds a new one of the kind%
\[
\widetilde{K}\left(  x,y\right)  =a\left(  x\right)  K\left(  x,y\right)
b\left(  y\right)
\]
where $a,b$ are suitable cutoff functions both supported in $B\left(
\overline{x},R_{0}\right)  .$ This $\widetilde{K}$ has the better property of
being defined in the whole $\Omega_{n+1}\times\Omega_{n+1}$ (except the
diagonal $x=y$); the integral operator with kernel $\widetilde{K}$ can be the
right object to prove a local estimate, of $L^{p}\ $or $C^{\alpha}$ type. We
can use the H\"{o}lder continuous cutoff functions built in the previous
section to define a kernel $\widetilde{K}$ supported in $B\left(  \overline
{x},R\right)  \times B\left(  \overline{x},R\right)  ,$ and exploit the fact
that $B\left(  \overline{x},R\right)  $ is in turn essentially contained in a
space of homogeneous type (see Theorem \ref{Thm F}). Then, we would like to
apply to the singular or fractional integral defined by $\widetilde{K}$ some
existing results from the theory of spaces of homogeneous type. This requires
checking that $\widetilde{K}$ satisfies globally, in the space of homogeneous
type where we have embedded it, suitable properties: standard estimates,
cancellation properties and so on. The following preliminary construction and
results serve to this aim. Moreover, in view of the commutator theorems we are
going to prove, we want to further shrink the support of $\widetilde{K},$ if
necessary. This is the reason why we introduce a second variable radius
$R<R_{0}$. We keep assuming that $\left(  \Omega,\left\{  \Omega_{n}\right\}
_{n=1}^{\infty},\rho,\mu\right)  $ be a locally homogeneous space. Moreover,
we make the following:

\bigskip

\textbf{Assumption (H7). }For fixed $\Omega_{n},\Omega_{n+1},$ and a fixed
ball $B\left(  \overline{x},R_{0}\right)  ,$ with $\overline{x}\in\Omega_{n}$
and $R_{0}<2\varepsilon_{n}$ (hence $B\left(  \overline{x},R_{0}\right)
\subset\Omega_{n+1}$), let $K\left(  x,y\right)  $ be a measurable function
defined for $x,y\in B\left(  \overline{x},R_{0}\right)  $, $x\neq y$. Let
$R>0$ be any number satisfying%
\begin{equation}
cR\leq R_{0}\label{R R_0}%
\end{equation}
for some $c>1$ which will be chosen in the proof of the next Proposition; let
$a,b\in C_{0}^{\alpha}\left(  \Omega_{n+1}\right)  ,$ $B\left(  \overline
{x},c_{1}R\right)  \prec a\prec B\left(  \overline{x},c_{2}R\right)  ,$
$B\left(  \overline{x},c_{3}R\right)  \prec b\prec B\left(  \overline{x}%
,c_{4}R\right)  $ (see Proposition \ref{Prop cutoff}) for some fixed constants
$c_{i}\in\left(  0,1\right)  ,$ $i=1,...,4.$ The new kernel%
\begin{equation}
\widetilde{K}\left(  x,y\right)  =a\left(  x\right)  K\left(  x,y\right)
b\left(  y\right)  \label{k tilde}%
\end{equation}
can be considered defined in the whole $\Omega_{n+1}\times\Omega
_{n+1}\setminus\left\{  x=y\right\}  $. Then:

\begin{proposition}
\label{Prop check assumptions}Under assumption (H7) we have:

(i) Assume $K$ satisfies for some $\nu\in\lbrack0,1)$ the following
\emph{standard estimates}:%
\begin{equation}
\left\vert K\left(  x,y\right)  \right\vert \leq\frac{A\rho\left(  x,y\right)
^{\nu}}{\mu\left(  B\left(  x,\rho\left(  x,y\right)  \right)  \right)  }
\label{standard 1}%
\end{equation}
for $x,y\in B\left(  \overline{x},R_{0}\right)  ,$ $x\neq y,$ and%
\begin{equation}
\left\vert K(x_{0},y)-K(x,y)\right\vert +\left\vert K(y,x_{0}%
)-K(y,x)\right\vert \leq\frac{B\rho\left(  x_{0},y\right)  ^{\nu}}{\mu
(B(x_{0},\rho(x_{0},y)))}\left(  \frac{\rho(x_{0},x)}{\rho(x_{0},y)}\right)
^{\beta} \label{standard 2}%
\end{equation}
for any $x_{0},x,y\in B\left(  \overline{x},R_{0}\right)  $ with $\rho
(x_{0},y)>M\rho(x_{0},x)$, some $\beta>0,M>1$. ($M\geq2B_{n+1}$, so that
condition $\rho(x_{0},y)>M\rho(x_{0},x)$ implies the comparability of
$\rho(x_{0},y)$ and $\rho\left(  x,y\right)  $).

Then $\widetilde{K}$ satisfies the same bound (\ref{standard 1}) for any
$x,y\in\Omega_{n+1},x\neq y$ and a bound (\ref{standard 2}) (with a different
constant $B^{\prime}$) for any $x_{0},x,y\in\Omega_{n+1}$, with $\rho
(x_{0},y)>M\rho(x_{0},x),$ provided $\alpha\geq\beta$ (where $\alpha$ is the
H\"{o}lder exponent related to the cutoff functions defining $\widetilde{K}$);
the new constant $B^{\prime}$ depends on $A,B$ and $n$ (but not on $R$).

(ii) Assume $K$ satisfies (\ref{standard 1}) with $\nu=0$ and the following
\emph{cancellation property}:

there exists $C>0$ such that for a.e. $x\in B\left(  \overline{x}%
,R_{0}\right)  $ and every $\varepsilon_{1},\varepsilon_{2}$ such that
$0<\varepsilon_{1}<\varepsilon_{2}$ and $B_{\rho^{\prime}}\left(
x,\varepsilon_{2}\right)  \subset\Omega_{n+1}$%
\begin{equation}
\left\vert \int_{\Omega_{n+1},\varepsilon_{1}<\rho^{\prime}(x,y)<\varepsilon
_{2}}K(x,y)\,d\mu(y)\right\vert +\left\vert \int_{\Omega_{n+1},\varepsilon
_{1}<\rho^{\prime}(x,z)<\varepsilon_{2}}K(z,x)\,d\mu(z)\right\vert \leq C,
\label{standard 3}%
\end{equation}
where $\rho^{\prime}$ is any quasidistance equivalent to $\rho$ in
$\Omega_{n+1}$ and $B_{\rho^{\prime}}$ denotes $\rho^{\prime}$-balls.

Then $\widetilde{K}$ satisfies a similar cancellation property (with a
different constant $C^{\prime}$) for a.e. $x\in\Omega_{n+1},0<\varepsilon
_{1}<\varepsilon_{2}<\infty.$ The new constant $C^{\prime}$ depends on $A,C$
and $n$ (but not on $R$).

The same is true if, in the condition (\ref{standard 3}), we replace the
integration over $\Omega_{n+1}$ with the integration over any measurable set
containing $B\left(  \overline{x},R\right)  .$

(iii) Assume $K$ satisfies the bound (i) and the following \emph{convergence
condition}: for a.e. $x\in B\left(  \overline{x},R_{0}\right)  $ such that
$B_{\rho^{\prime}}\left(  x,R\right)  \subset\Omega_{n+1}$ there exists%
\[
h_{R}\left(  x\right)  \equiv\lim_{\varepsilon\rightarrow0}\int_{\Omega
_{n+1},\varepsilon<\rho^{\prime}(x,y)<R}K(x,y)d\mu(y),
\]
where $\rho^{\prime}$ is any quasidistance equivalent to $\rho$ in
$\Omega_{n+1}$.

Then for a.e. $x\in\Omega_{n+1}$, there exists%
\[
\widetilde{h}\left(  x\right)  \equiv\lim_{\varepsilon\rightarrow0}%
\int_{\Omega_{n+1},\rho^{\prime}(x,y)>\varepsilon}\widetilde{K}(x,y)\,d\mu
(y).
\]

\end{proposition}

\begin{remark}
The presence of a function $\rho^{\prime}$ possibly different from $\rho$ (but
equivalent to it) in conditions (ii)-(iii) adds flexibility to the theory: it
is sometimes easier to check these conditions for a $\rho^{\prime}$ different
from $\rho.$ For instance, when dealing with local estimates for operators
structured on H\"{o}rmander's vector fields, typically $\rho$ will be the
Carnot-Carath\'{e}odory distance induced by the vector fields, while
$\rho^{\prime}$ will be the quasidistance defined by Rothschild-Stein in
\cite{RS}.
\end{remark}

\begin{proof}
The first part of (i) is obvious. To prove the second part, let us write, for
$x_{0},x,y\in\Omega_{n+1}$, $\rho(x_{0},y)>M\rho(x_{0},x)$:%
\begin{align*}
\left\vert \widetilde{K}(x_{0},y)-\widetilde{K}(x,y)\right\vert  &
\leq\left\vert a\left(  x_{0}\right)  K\left(  x_{0},y\right)  b\left(
y\right)  -a\left(  x_{0}\right)  K\left(  x,y\right)  b\left(  y\right)
\right\vert \\
&  +\left\vert a\left(  x_{0}\right)  K\left(  x,y\right)  b\left(  y\right)
-a\left(  x\right)  K\left(  x,y\right)  b\left(  y\right)  \right\vert \\
&  =I+II.
\end{align*}
By (\ref{standard 2}),
\begin{equation}
I\leq\left\vert a\left(  x_{0}\right)  b\left(  y\right)  \right\vert
\frac{B\rho\left(  x_{0},y\right)  ^{\nu}}{\mu(B(x_{0},\rho(x_{0},y)))}\left(
\frac{\rho(x_{0},x)}{\rho(x_{0},y)}\right)  ^{\beta} \label{Standard I}%
\end{equation}
when $x_{0},x,y\in B\left(  \overline{x},R_{0}\right)  $. Since the quantity
$\left\vert a\left(  x_{0}\right)  b\left(  y\right)  \right\vert $ does not
vanish only if $x_{0},y\in B\left(  \overline{x},R\right)  $, it is enough to
consider what happens when $x_{0},y\in B\left(  \overline{x},R\right)  $ and
$x\notin B\left(  \overline{x},R\right)  .$ We have%
\begin{align*}
\rho\left(  x,\overline{x}\right)   &  \leq B_{n+1}\left(  \rho\left(
x,x_{0}\right)  +\rho\left(  x_{0},\overline{x}\right)  \right)  \leq
B_{n+1}\left(  \frac{1}{M}\rho\left(  x_{0},y\right)  +R\right) \\
&  \leq B_{n+1}\left(  \frac{1}{M}B_{n+1}\left(  \rho\left(  x_{0}%
,\overline{x}\right)  +\rho\left(  \overline{x},y\right)  \right)  +R\right)
\\
&  \leq B_{n+1}\left(  \frac{1}{M}B_{n+1}2R+R\right)  \leq2B_{n+1}R
\end{align*}
(by our assumption $M\geq2B_{n+1}$). Hence, if in Assumption (H7) we take
\begin{equation}
cR\leq R_{0}\text{ with }c>2B_{n+1}, \label{c H7}%
\end{equation}
we have $\rho\left(  x,\overline{x}\right)  \leq R_{0},$ and (\ref{Standard I}%
) holds for any $x_{0},x,y\in\Omega_{n+1}$ with $\rho(x_{0},y)>M\rho(x_{0},x)$.

Now,
\[
II=\left\vert a\left(  x_{0}\right)  -a\left(  x\right)  \right\vert
\left\vert K\left(  x,y\right)  b\left(  y\right)  \right\vert
\]
by Proposition \ref{Prop cutoff} and (\ref{standard 1}), for $x,y\in B\left(
\overline{x},R_{0}\right)  ,$%
\begin{align}
&  \leq c\left(  \frac{\rho(x_{0},x)}{R}\right)  ^{\alpha}\left\vert K\left(
x,y\right)  b\left(  y\right)  \right\vert \nonumber\\
&  \leq c\frac{B\rho\left(  x_{0},y\right)  ^{\nu}}{\mu(B(x_{0},\rho
(x_{0},y)))}\left(  \frac{\rho(x_{0},x)}{R}\right)  ^{\alpha} \label{rho/R}%
\end{align}
since $\rho\left(  x_{0},y\right)  $ is comparable to $\rho\left(  x,y\right)
.$

The term $II\ $does not vanish only if $y\in B\left(  \overline{x},R\right)  $
and $x$ or $x_{0}$ belongs to $B\left(  \overline{x},R\right)  .$

If $y,x_{0}\in B\left(  \overline{x},R\right)  $ then $\rho\left(
x_{0},y\right)  \leq2B_{n+1}R.$ On the other hand, condition $\rho
(x_{0},y)>M\rho(x_{0},x)$ with $M\geq2B_{n+1}$ implies
\[
\rho\left(  x_{0},y\right)  \leq2B_{n+1}\rho\left(  x,y\right)  .
\]
Hence if $y,x\in B\left(  \overline{x},R\right)  $ then $\rho\left(
x,y\right)  \leq2B_{n+1}R$ and $\rho\left(  x_{0},y\right)  \leq\left(
2B_{n+1}\right)  ^{2}R.$ So, in any case $\rho\left(  x_{0},y\right)  \leq
c_{1}R,$ and (\ref{rho/R}) gives%
\begin{align*}
II  &  \leq c\frac{B\rho\left(  x,y\right)  ^{\nu}}{\mu(B(x_{0},\rho
(x,y)))}\left(  \frac{\rho(x_{0},x)}{\rho(x_{0},y)}\right)  ^{\alpha}\\
&  \leq c\frac{B\rho\left(  x_{0},y\right)  ^{\nu}}{\mu(B(x_{0},\rho
(x_{0},y)))}\left(  \frac{\rho(x_{0},x)}{\rho(x_{0},y)}\right)  ^{\beta}%
\end{align*}
for any $\beta\leq\alpha,x,y\in B\left(  \overline{x},R_{0}\right)  .$ It is
now enough to check what happens for $y,x_{0}\in B\left(  \overline
{x},R\right)  $ and $x\notin B\left(  \overline{x},R_{0}\right)  .$ Reasoning
like above, the conditions $y,x_{0}\in B\left(  \overline{x},R\right)  $ imply
$\rho\left(  x,\overline{x}\right)  \leq2B_{n+1}R<R_{0}$ by (\ref{c H7}),
hence $x\notin B\left(  \overline{x},R_{0}\right)  $ simply cannot happen.

To prove (ii), let $x\in\Omega_{n+1}$ and consider, for any $\varepsilon
_{2}>\varepsilon_{1}>0,$%
\[
A\equiv\int_{\Omega_{n+1},\varepsilon_{1}<\rho^{\prime}(x,y)<\varepsilon_{2}%
}\widetilde{K}(x,y)\,d\mu(y)=a\left(  x\right)  \int_{\Omega_{n+1}%
,\varepsilon_{1}<\rho^{\prime}(x,y)<\varepsilon_{2}}K(x,y)b\left(  y\right)
\,d\mu(y).
\]
This quantity does not vanish only if $x\in B\left(  \overline{x},R\right)  ;$
since also $b\left(  y\right)  $ does not vanish only for $y\in B\left(
\overline{x},R\right)  ,$ the integrand does not vanish only if $\rho^{\prime
}\left(  x,y\right)  \leq c^{\prime}R$ for some $c^{\prime}$. Choose the
number $c$ in (\ref{R R_0}) so that $cR\leq R_{0}$ implies $c^{\prime}R\leq
R_{0}$ and $B\left(  x,c^{\prime}R\right)  \subset\Omega_{n+1}$. Then%
\begin{align*}
A &  =a\left(  x\right)  \int_{\Omega_{n+1},\varepsilon_{1}<\rho^{\prime
}(x,y)<\min\left(  \varepsilon_{2},c^{\prime}R\right)  }K(x,y)b\left(
y\right)  \,d\mu(y)\\
&  =a\left(  x\right)  b\left(  x\right)  \int_{\Omega_{n+1},\varepsilon
_{1}<\rho^{\prime}(x,y)<\min\left(  \varepsilon_{2},c^{\prime}R\right)
}K(x,y)\,d\mu(y)+\\
&  +\int_{\Omega_{n+1},\varepsilon_{1}<\rho^{\prime}(x,y)<\min\left(
\varepsilon_{2},c^{\prime}R\right)  }K(x,y)\,\left[  b\left(  y\right)
-b\left(  x\right)  \right]  d\mu(y)\equiv A_{1}+A_{2}.
\end{align*}
Now by (\ref{standard 3}) we can bound%
\[
\left\vert A_{1}\right\vert \leq C\left\vert a\left(  x\right)  b\left(
x\right)  \right\vert \leq C
\]
while%
\begin{align*}
\left\vert A_{2}\right\vert  &  \leq\int_{\Omega_{n+1},\varepsilon_{1}%
<\rho^{\prime}(x,y)<\min\left(  \varepsilon_{2},c^{\prime}R\right)
}\left\vert K(x,y)\right\vert c\left(  \frac{\rho\left(  x,y\right)  }%
{R}\right)  ^{\alpha}d\mu(y)\\
&  \leq\frac{c}{R^{\alpha}}\int_{\Omega,\rho(x,y)<c^{\prime\prime}R}%
\frac{A\rho\left(  x,y\right)  ^{\alpha}}{\mu\left(  B\left(  x,\rho\left(
x,y\right)  \right)  \right)  }d\mu\left(  y\right)  \leq\frac{c}{R^{\alpha}%
}c^{\prime\prime\prime}R^{\alpha}\equiv C^{\prime}.
\end{align*}
In the last inequality we have applied a standard estimate in spaces of
homogeneous type (since the integral is extended to a ball centered at a point
of $\Omega_{n+1}$ and contained in $\Omega_{n+2}$ we can apply the local
doubling condition):%
\begin{align*}
&  \int_{\Omega,\rho(x,y)<c^{\prime\prime}R}\frac{\rho\left(  x,y\right)
^{\alpha}}{\mu\left(  B\left(  x,\rho\left(  x,y\right)  \right)  \right)
}d\mu\left(  y\right)  \\
&  \leq\sum_{k=0}^{\infty}\int_{\Omega_{n+1},\frac{c^{\prime\prime}R}{2^{k+1}%
}\leq\rho(x,y)<\frac{c^{\prime\prime}R}{2^{k}}}\frac{\rho(x,y)^{\alpha}}%
{\mu\left(  B\left(  x,\rho\left(  x,y\right)  \right)  \right)  }d\mu(y)\\
&  \leq\underset{k=0}{\overset{\infty}{\sum}}\left(  \frac{c^{\prime\prime}%
R}{2^{k}}\right)  ^{\alpha}\frac{\mu\left(  B\left(  x,\frac{c^{\prime\prime
}R}{2^{k}}\right)  \right)  }{\mu\left(  B\left(  x,\frac{c^{\prime\prime}%
R}{2^{k+1}}\right)  \right)  }\leq\underset{k=0}{\overset{\infty}{\sum}%
}C_{n+1}\left(  \frac{c^{\prime\prime}R}{2^{k}}\right)  ^{\alpha}%
=c^{\prime\prime\prime}R^{\alpha}.
\end{align*}

To prove (iii), let us consider, for $x\in\Omega_{n+1}$ and $0<\varepsilon
_{1}<\varepsilon_{2},$%
\begin{align*}
&  \int_{\Omega_{n+1},\rho^{\prime}(x,y)>\varepsilon_{2}}\widetilde
{K}(x,y)\,d\mu(y)-\int_{\Omega_{n+1},\rho^{\prime}(x,y)>\varepsilon_{1}%
}\widetilde{K}(x,y)\,d\mu(y)\\
&  =a\left(  x\right)  \int_{\Omega_{n+1},\varepsilon_{1}<\rho^{\prime
}(x,y)\leq\varepsilon_{2}}K(x,y)b\left(  y\right)  d\mu(y)\equiv A\left(
\varepsilon_{1},\varepsilon_{2},x\right)  .
\end{align*}
The quantity $A\left(  \varepsilon_{1},\varepsilon_{2},x\right)  $ does not
vanish only if $x\in B\left(  \overline{x},R\right)  ;$ for this $x$ and $R$
small enough we have $B\left(  x,R\right)  \subset\Omega_{n+1}$, hence we can
write%
\begin{align*}
A\left(  \varepsilon_{1},\varepsilon_{2},x\right)   &  =a\left(  x\right)
b\left(  x\right)  \int_{\Omega_{n+1},\varepsilon_{1}<\rho^{\prime}%
(x,y)\leq\varepsilon_{2}}K(x,y)d\mu(y)\\
&  +a\left(  x\right)  \int_{\Omega_{n+1},\varepsilon_{1}<\rho^{\prime
}(x,y)\leq\varepsilon_{2}}K(x,y)\left[  b\left(  y\right)  -b\left(  x\right)
\right]  d\mu(y)\\
&  \equiv A_{1}\left(  \varepsilon_{1},\varepsilon_{2},x\right)  +A_{2}\left(
\varepsilon_{1},\varepsilon_{2},x\right)
\end{align*}
and, by our assumption on $K,$%
\[
\lim_{\varepsilon_{1},\varepsilon_{2}\rightarrow0}A_{1}\left(  \varepsilon
_{1},\varepsilon_{2},x\right)  =0.
\]
On the other hand, reasoning as above, we get%
\[
\left\vert A_{2}\left(  \varepsilon_{1},\varepsilon_{2},x\right)  \right\vert
\leq\frac{c}{R^{\alpha}}\int_{\Omega,\rho(x,y)<c^{\prime\prime}\varepsilon
_{2}}\frac{A\rho\left(  x,y\right)  ^{\alpha}}{\mu\left(  B\left(
x,\rho\left(  x,y\right)  \right)  \right)  }d\mu\left(  y\right)  \leq
c\left(  \frac{\varepsilon_{2}}{R}\right)  ^{\alpha}%
\]
which also vanishes for $\varepsilon_{2}\rightarrow0.$ So the desired limit exists.
\end{proof}

\begin{theorem}
[$L^{p}$ and $C^{\eta}$ estimates for singular integrals]%
\label{Theorem L^p C^eta}Let $K,\widetilde{K}$ be as in Assumption (H7), with
$K$ satisfying the standard estimates (i) with $\nu=0,$ the cancellation
property (ii) and the convergence condition (iii) stated in Proposition
\ref{Prop check assumptions}. If%
\[
Tf\left(  x\right)  =\lim_{\varepsilon\rightarrow0}\int_{B\left(  \overline
{x},R\right)  ,\rho^{\prime}(x,y)>\varepsilon}\widetilde{K}(x,y)f\left(
y\right)  d\mu(y),
\]
then for any $p\in\left(  1,\infty\right)  $%
\[
\left\Vert Tf\right\Vert _{L^{p}\left(  B\left(  \overline{x},R\right)
\right)  }\leq c\left\Vert f\right\Vert _{L^{p}\left(  B\left(  \overline
{x},R\right)  \right)  }.
\]
The constant $c$ depends on $p,n$ and the constants of $K$ involved in the
assumptions (but not on $R$).

Moreover, $T$ satisfies a weak 1-1 estimate:%
\[
\mu\left(  \left\{  x\in B\left(  \overline{x},R\right)  :\left\vert Tf\left(
x\right)  \right\vert >t\right\}  \right)  \leq\frac{c}{t}\left\Vert
f\right\Vert _{L^{1}\left(  B\left(  \overline{x},R\right)  \right)  }\text{
for any }t>0.
\]

Assume that, in addition, the kernel $K$ satisfies the condition
\begin{equation}
\widetilde{h}\left(  x\right)  \equiv\lim_{\varepsilon\rightarrow0}\int
_{\rho^{\prime}(x,y)>\varepsilon}\widetilde{K}(x,y)d\mu(y)\in C^{\gamma
}\left(  \Omega_{n+1}\right)  \label{h tilde C^gamma}%
\end{equation}
for some $\gamma>0$ (where $\rho^{\prime}$ is the same appearing in the
assumed convergence condition (iii)). Then%
\begin{equation}
\left\Vert Tf\right\Vert _{C^{\eta}\left(  B\left(  \overline{x},R\right)
\right)  }\leq c\left\Vert f\right\Vert _{C^{\eta}\left(  B\left(
\overline{x},HR\right)  \right)  } \label{C_eta}%
\end{equation}
for any positive $\eta<\min\left(  \alpha,\beta,\gamma\right)  $ and some
constant $H>1$ independent of $R$. (Recall that $\alpha$ is the H\"{o}lder
exponent related to the cutoff functions defining $\widetilde{K},\beta$
appears in the standard estimates (i) and $\gamma$ is the number in
(\ref{h tilde C^gamma})).

The constant $c$ depends on $\eta,n,R,$ the constants involved in the
assumptions on $K,$ and the $C^{\gamma}$ norm of $\widetilde{h}.$
\end{theorem}

\begin{proof}
By Theorem \ref{Thm F} there exists a space of homogeneous type $F$ such that%
\[
B\left(  \overline{x},R\right)  \setminus E\subset F\subset\Omega_{n+1}%
\]
and the doubling constant of $F$ only depends on $n.$ By our assumptions on
$K$ and Proposition \ref{Prop check assumptions}, the operator $T$ satisfies
all the assumptions of Theorem \ref{sin lp} (about singular integrals in
spaces of homogeneous type), so that
\[
\left\Vert Tf\right\Vert _{L^{p}\left(  B\left(  \overline{x},R\right)
\right)  }\leq\left\Vert Tf\right\Vert _{L^{p}\left(  F\right)  }\leq
c\left\Vert f\right\Vert _{L^{p}\left(  F\right)  }%
\]
with $c$ depending on $p,n$ and the constants involved in the assumptions
about $K$. Applying the inequality to $f\in L^{p}\left(  B\left(  \overline
{x},R\right)  \right)  $ (having set $f=0$ outside this ball), we get%
\[
\left\Vert Tf\right\Vert _{L^{p}\left(  B\left(  \overline{x},R\right)
\right)  }\leq c\left\Vert f\right\Vert _{L^{p}\left(  B\left(  \overline
{x},R\right)  \right)  }.
\]
The same argument gives the weak 1-1 estimate for $T$.

For the $C^{\eta}$ case a similar argument applies; we now apply Theorem
\ref{sin holder} and get%
\[
\left\Vert Tf\right\Vert _{C^{\eta}\left(  F\right)  }\leq c\left\Vert
f\right\Vert _{C^{\eta}\left(  F\right)  }%
\]
with $c$ depending on $\eta,n,$ the constants involved in the assumptions
about $K,$ the $C^{\gamma}$ norm of $\widetilde{h},$ and also diam$F$, that is
$R$. Moreover, $B\left(  \overline{x},R\right)  \subset\overline{F}$ hence
\[
\left\Vert Tf\right\Vert _{C^{\eta}\left(  B\left(  \overline{x},R\right)
\right)  }\leq\left\Vert Tf\right\Vert _{C^{\eta}\left(  \overline{F}\right)
}=\left\Vert Tf\right\Vert _{C^{\eta}\left(  F\right)  }\leq c\left\Vert
f\right\Vert _{C^{\eta}\left(  F\right)  }.
\]
A difference with the $L^{p}$ case is that now we cannot set $f=0$ outside the
ball $B\left(  \overline{x},R\right)  $ preserving its H\"{o}lder continuity,
therefore we can just write%
\[
\left\Vert Tf\right\Vert _{C^{\eta}\left(  B\left(  \overline{x},R\right)
\right)  }\leq c\left\Vert f\right\Vert _{C^{\eta}\left(  B\left(
\overline{x},HR\right)  \right)  }%
\]
since, for some $H>1$ independent of $R$, we have $F\subset B\left(
\overline{x},HR\right)  $, as seen in the proof of Theorem \ref{Thm F}.
\end{proof}

\begin{remark}
[Estimates for $C_{0}^{\eta}$ functions]\label{remark local holder}In the
applications of this theory to local a priori estimates for PDEs, the function
$f$ is usually compactly supported, so that we can apply (\ref{C_eta}) to
$f\in C_{0}^{\eta}\left(  B\left(  \overline{x},R\right)  \right)  ,$ getting
the more appealing inequality%
\[
\left\Vert Tf\right\Vert _{C^{\eta}\left(  B\left(  \overline{x},R\right)
\right)  }\leq c\left\Vert f\right\Vert _{C^{\eta}\left(  B\left(
\overline{x},R\right)  \right)  }.
\]
Moreover, applying this inequality to functions $f\in C_{0}^{\eta}\left(
B\left(  \overline{x},r\right)  \right)  $ with $r<R$ we can get a a bound%
\[
\left\Vert Tf\right\Vert _{C^{\eta}\left(  B\left(  \overline{x},r\right)
\right)  }\leq c\left\Vert f\right\Vert _{C^{\eta}\left(  B\left(
\overline{x},r\right)  \right)  }%
\]
with $c$ depending on $R$ but not on $r$.
\end{remark}

\begin{remark}
[Checking assumption (\ref{h tilde C^gamma}) ]Assumption
(\ref{h tilde C^gamma}) can be the most troublesome to check in concrete
applications (apart from classical cases in which $\widetilde{h}$ is zero, or
is constant). In some applications to subelliptic equations, the kernel
$\widetilde{K}(x,y)$ happens to be a perturbation of a simpler kernel which
has vanishing integral over spherical shells; in these cases, one can prove
that the limit%
\[
\lim_{\varepsilon\rightarrow0}\int_{\rho^{\prime}(x,y)>\varepsilon}%
\widetilde{K}(x,y)d\mu(y)
\]
equals to the integral of a nonsingular kernel satisfying standard estimates
(\ref{standard 1})-(\ref{standard 2}) for some $\nu>0.$ It is then helpful to
recall that such an integral \emph{always }belongs to a H\"{o}lder space, as
will follow from Theorem \ref{Thm frac C^eta}, since it can be regarded as
$T\left(  1\right)  $, where the constant $1$ is H\"{o}lder continuous and $T$
is a fractional integral..
\end{remark}

\begin{theorem}
[$L^{p}-L^{q}$ estimate for fractional integrals]\label{frac lp-lq}Let
$K,\widetilde{K}$ be as in Assumption (H7), with $K$ satisfying the growth
condition
\begin{equation}
0\leq K\left(  x,y\right)  \leq\frac{c}{\mu\left(  B\left(  x,\rho\left(
x,y\right)  \right)  \right)  ^{1-\nu}} \label{Standard 1'}%
\end{equation}
for some $\nu\in\left(  0,1\right)  ,$ $c>0,$ any $x,y\in B\left(
\overline{x},R_{0}\right)  ,$ $x\neq y.$ If%
\[
I_{\nu}f\left(  x\right)  =\int_{B\left(  \overline{x},R\right)  }%
\widetilde{K}(x,y)f\left(  y\right)  d\mu(y)
\]
then, for any $p\in\left(  1,\frac{1}{\nu}\right)  ,\frac{1}{q}=\frac{1}%
{p}-\nu$ there exists $c\ $such that%
\[
\left\Vert I_{\nu}f\right\Vert _{L^{q}\left(  B\left(  \overline{x},R\right)
\right)  }\leq c\left\Vert f\right\Vert _{L^{p}\left(  B\left(  \overline
{x},R\right)  \right)  }%
\]
for any $f\in L^{p}\left(  B\left(  \overline{x},R\right)  \right)  .$The
constant $c$ depends on $p,n,$ and the constants of $K$ involved in the
assumptions (but not on $R$).
\end{theorem}

\begin{proof}
This theorem follows from the analog result which holds in spaces of
homogeneous type, that is Theorem \ref{frac lp}, by a similar argument to that
used in the proof of Theorem \ref{Theorem L^p C^eta}.
\end{proof}

The analog $C^{\eta}$ estimate for fractional integrals is better stated under
slightly different assumptions on the kernel. In the applications of the
theory that we have in mind, where the measure of a ball is equivalent to a
fixed power of the radius, both the theorems will be applicable.

\begin{theorem}
[$C^{\eta}$ estimate for fractional integrals]\label{Thm frac C^eta}Let
$K,\widetilde{K}$ be as in Assumption (H7), with $K$ satisfying
(\ref{standard 1}) and (\ref{standard 2}) for some $\nu\in\left(  0,1\right)
,\beta>0$. If%
\[
I_{\nu}f\left(  x\right)  =\int_{B\left(  \overline{x},R\right)  }%
\widetilde{K}(x,y)f\left(  y\right)  d\mu(y),
\]
then, for any $\eta<\min\left(  \alpha,\beta\,,\nu\right)  $%
\[
\left\Vert I_{\nu}f\right\Vert _{C^{\eta}\left(  B\left(  \overline
{x},R\right)  \right)  }\leq c\left\Vert f\right\Vert _{C^{\eta}\left(
B\left(  \overline{x},HR\right)  \right)  }.
\]
The constant $c$ depends on $\eta,n,R$ and the constants of $K$ involved in
the assumptions; the number $H$ only depends on $n$.
\end{theorem}

Reasoning as in Remark \ref{remark local holder}, we can also say that for
functions $f\in C_{0}^{\eta}\left(  B\left(  \overline{x},r\right)  \right)  $
with $r<R$ the following bound holds%
\[
\left\Vert I_{\nu}f\right\Vert _{C^{\eta}\left(  B\left(  \overline
{x},r\right)  \right)  }\leq c\left\Vert f\right\Vert _{C^{\eta}\left(
B\left(  \overline{x},r\right)  \right)  }%
\]
with $c$ depending on $R$ but not on $r$.

\begin{proof}
This theorem follows from Proposition \ref{Prop check assumptions} and the
analog result which holds in spaces of homogeneous type, that is Theorem
\ref{Thm frac C^alfa}, by an argument similar to that used in the proof of the
$C^{\eta}$ case in Theorem \ref{Theorem L^p C^eta}.
\end{proof}

\section{Local and global $BMO$ and $VMO$ spaces\label{Section BMO}}

Let $\left(  \Omega,\left\{  \Omega_{n}\right\}  _{n=1}^{\infty},\rho
,\mu\right)  $ be a locally homogeneous space.

\begin{definition}
[Local $BMO$ and $VMO$ spaces]For any function $u\in L^{1}\left(  \Omega
_{n+1}\right)  $, and $r>0$, with $r\leq\varepsilon_{n},$ set%
\[
\eta_{u,\Omega_{n},\Omega_{n+1}}^{\ast}(r)=\sup_{t\leq r}\sup_{x_{0}\in
\Omega_{n}}\frac{1}{\mu\left(  B\left(  x_{0},t\right)  \right)  }%
\int_{B\left(  x_{0},t\right)  }|u(x)-u_{B}|\,d\mu\left(  x\right)  ,
\]
where $u_{B}=\mu(B\left(  x_{0},t\right)  )^{-1}\int_{B\left(  x_{0},t\right)
}u$. We say that $u\in BMO_{loc}\left(  \Omega_{n},\Omega_{n+1}\right)  $ if%
\[
\left\Vert u\right\Vert _{BMO_{loc}\left(  \Omega_{n},\Omega_{n+1}\right)
}=\sup_{r\leq\varepsilon_{n}}\eta_{u,\Omega_{n},\Omega_{n+1}}^{\ast}\left(
r\right)  <\infty.
\]
We say that $u\in VMO_{loc}\left(  \Omega_{n},\Omega_{n+1}\right)  $ if $u\in
BMO_{loc}\left(  \Omega_{n},\Omega_{n+1}\right)  $ and
\[
\eta_{u,\Omega_{n},\Omega_{n+1}}^{\ast}(r)\rightarrow0\text{ as }%
r\rightarrow0.
\]
The function $\eta_{u,\Omega_{n},\Omega_{n+1}}^{\ast}$ will be called $VMO$
local modulus of $u$ in $\left(  \Omega_{n},\Omega_{n+1}\right)  $.
\end{definition}

Note that in the previous definition we integrate $u$ over balls centered at
points of $\Omega_{n}$ and enclosed in $\Omega_{n+1}$. This is a fairly
natural definition if we want to avoid integrating over the \emph{intersection
}$B\left(  x_{0},t\right)  \cap\Omega_{n}.$ We will need also the following standard

\begin{definition}
[$BMO$ and $VMO$ spaces over a homogeneous space]Let $S$ be a subset of
$\Omega$ which is a space of homogeneous type ($S$ can be a single cube
$Q_{\alpha}^{k},$ or the set $F$ built in Theorem \ref{Thm F}, or the whole $%
{\displaystyle\bigcup\limits_{\alpha\in I_{k}}}
Q_{\alpha}^{k}$). For any function $u\in L^{1}\left(  S\right)  $ and $r>0,$
set%
\[
\eta_{u,S}(r)=\sup_{t\leq r}\sup_{x_{0}\in S}\frac{1}{\mu\left(  B\left(
x_{0},t\right)  \cap S\right)  }\int_{B\left(  x_{0},t\right)  \cap
S}|u(x)-u_{B\cap S}|\,d\mu\left(  x\right)  ,
\]
where $u_{B\cap S}=\mu(B\left(  x_{0},t\right)  \cap S)^{-1}\int_{B\left(
x_{0},t\right)  \cap S}u$. We say that $u\in BMO\left(  S\right)  $ if%
\[
\left\Vert u\right\Vert _{BMO\left(  S\right)  }=\sup_{r>0}\eta_{u,S}\left(
r\right)  <\infty.
\]
We say that $u\in VMO\left(  S\right)  $ if $u\in BMO\left(  S\right)  $ and
$\eta_{u,S}(r)\rightarrow0$ as $r\rightarrow0$. The function $\eta_{u,S}$ will
be called $VMO$ modulus of $u$ in $S$.
\end{definition}

The useful link between the two notions of $BMO$ is contained in the next
Proposition. Here we have to consider families of dyadic cubes adapted to
different sets $\Omega_{n},$ hence we need to add an extra index to our symbols.

\begin{proposition}
\label{Prop BMO loc BMO}For fixed $n$ and $\overline{x}\in\Omega_{n},$ let
$B\left(  \overline{x},R\right)  ,F$ be as in Theorem \ref{Thm F} (recall that
$F\subset\Omega_{n+2}$). Let $f\in BMO_{loc}\left(  \Omega_{n+2},\Omega
_{n+3}\right)  ,$ then
\begin{equation}
\left\Vert f\right\Vert _{BMO\left(  F\right)  }\leq c\eta_{f,\Omega
_{n+2},\Omega_{n+3}}^{\ast}\left(  cR\right)  \leq c\left\Vert f\right\Vert
_{BMO_{loc}\left(  \Omega_{n+2},\Omega_{n+3}\right)  } \label{BMO eta}%
\end{equation}
for a constant $c$ depending on $n$ but independent of $R$. In particular,
given $f\in VMO_{loc}\left(  \Omega_{n+2},\Omega_{n+3}\right)  ,$ the norm
$\left\Vert f\right\Vert _{BMO\left(  F\right)  }$ can be taken as small as we
want, for fixed $n$ and $R$ small enough.
\end{proposition}

\begin{proof}
The second inequality holds by definition, so let us prove the first. With the
notation used in the proof of Theorem \ref{Thm F}, let $x\in F,$ that is $x\in
Q_{\beta}^{\left(  n+1\right)  ,k}$ for some $Q_{\beta}^{\left(  n+1\right)
,k}$ intersecting $B\left(  z_{\alpha}^{\left(  n\right)  ,k},h_{n}%
\delta_{\left(  n\right)  }^{k}\right)  .$ In particular, $x\in\Omega_{n+2}.$
Recall that $\delta_{\left(  n\right)  }^{k+1}<R\leq\delta_{\left(  n\right)
}^{k}.$ We want to bound, for any $r>0,$%
\[
I\equiv\frac{1}{\mu\left(  B\left(  x,r\right)  \cap F\right)  }\int_{B\left(
x,r\right)  \cap F}\left\vert f\left(  y\right)  -c\right\vert d\mu\left(
y\right)
\]
with $c$ to be chosen later. Let us distinguish the cases:

(i) $r\leq\delta_{\left(  n+1\right)  }^{k}$. Then
\[
\mu\left(  B\left(  x,r\right)  \cap F\right)  \geq\mu\left(  B\left(
x,r\right)  \cap Q_{\beta}^{\left(  n+1\right)  ,k}\right)  \geq c_{0,\left(
n+1\right)  }\mu\left(  B\left(  x,r\right)  \right)
\]
(see (\ref{lower bounds cubes})), hence choosing $c=f_{B\left(  x,r\right)  }$%
\begin{align*}
I  &  \leq\frac{c}{\mu\left(  B\left(  x,r\right)  \right)  }\int_{B\left(
x,r\right)  }\left\vert f\left(  y\right)  -f_{B\left(  x,r\right)
}\right\vert d\mu\left(  y\right) \\
&  \leq c\eta_{f,\Omega_{n+2},\Omega_{n+3}}^{\ast}\left(  \delta_{\left(
n+1\right)  }^{k}\right)  \leq c\eta_{f,\Omega_{n+2},\Omega_{n+3}}^{\ast
}\left(  c_{n}R\right)
\end{align*}
since $\delta_{\left(  n+1\right)  }\leq\delta_{\left(  n\right)  }$ and
$\delta_{\left(  n\right)  }^{k+1}<R.$

(ii) $r>\delta_{\left(  n+1\right)  }^{k}.$ Then
\[
\mu\left(  B\left(  x,r\right)  \cap F\right)  \geq\mu\left(  B\left(
x,r\right)  \cap Q_{\beta}^{\left(  n+1\right)  ,k}\right)  \geq c_{0,\left(
n+1\right)  }\mu\left(  Q_{\beta}^{\left(  n+1\right)  ,k}\right)
\]
(see (\ref{lower bounds cubes})), which in turn is equivalent to $\mu\left(
F\right)  $ because the two sets have comparable diameters and the first is
contained in the second. Here we can apply the local doubling condition for
balls of radius $\simeq\delta_{\left(  n+1\right)  }^{k}$ centered in
$\Omega_{n+1}$ and contained in $\Omega_{n+2}.$ Hence%
\[
I\leq\frac{c}{\mu\left(  F\right)  }\int_{F}\left\vert f\left(  y\right)
-c\right\vert d\mu\left(  y\right)  .
\]
In turn (see the last part of the proof of Theorem \ref{Thm F}), $F\subset
B\left(  \overline{x},j_{n}R\right)  $ with $\mu\left(  F\right)  $ comparable
to $\mu\left(  B\left(  \overline{x},j_{n}R\right)  \right)  $, therefore
choosing $c=f_{B\left(  \overline{x},j_{n}R\right)  }$ we get%
\[
I\leq c\eta_{f,\Omega_{n+2},\Omega_{n+3}}^{\ast}\left(  c_{n}R\right)
\]
and we are done.
\end{proof}

\section{Commutators of local singular and fractional integrals with BMO
functions\label{Section Commutator}}

\begin{theorem}
[Commutators of local singular integrals]\label{Thm commutator}Let
$K,\widetilde{K}$ be as in Assumption (H7), with $K$ satisfying the standard
estimates (i) with $\nu=0,$ the cancellation property (ii) and the convergence
condition (iii) (see Proposition \ref{Prop check assumptions}). If%
\[
Tf\left(  x\right)  =\lim_{\varepsilon\rightarrow0}\int_{B\left(  \overline
{x},R\right)  ,\rho^{\prime}(x,y)>\varepsilon}\widetilde{K}(x,y)f\left(
y\right)  d\mu(y)
\]
and, for $a\in BMO_{loc}\left(  \Omega_{n+2},\Omega_{n+3}\right)  ,$ we set%
\[
C_{a}f\left(  x\right)  =T\left(  af\right)  \left(  x\right)  -a\left(
x\right)  Tf\left(  x\right)  ,
\]
then for any $p\in\left(  1,\infty\right)  $ there exists $c>0$ such that%
\[
\left\Vert C_{a}f\right\Vert _{L^{p}\left(  B\left(  \overline{x},R\right)
\right)  }\leq c\left\Vert a\right\Vert _{BMO_{loc}\left(  \Omega_{n+2}%
,\Omega_{n+3}\right)  }\left\Vert f\right\Vert _{L^{p}\left(  B\left(
\overline{x},R\right)  \right)  }.
\]

Moreover, if $a\in VMO_{loc}\left(  \Omega_{n+2},\Omega_{n+3}\right)  $ for
any $\varepsilon>0$ there exists $r>0$ such that for any $f\in L^{p}\left(
B\left(  \overline{x},r\right)  \right)  $ we have%
\[
\left\Vert C_{a}f\right\Vert _{L^{p}\left(  B\left(  \overline{x},r\right)
\right)  }\leq\varepsilon\left\Vert f\right\Vert _{L^{p}\left(  B\left(
\overline{x},r\right)  \right)  }.
\]
The constant $c$ depends on $p,n$ and the constants of $K$ involved in the
assumptions (but not on $R$); the constant $r$ also depends on the
$VMO_{loc}\left(  \Omega_{n+2},\Omega_{n+3}\right)  $ modulus of $a$.
\end{theorem}

\begin{proof}
Proceeding like in the proof of Theorem \ref{Theorem L^p C^eta}, and with the
same meaning of the symbols, we prove, applying Theorems \ref{Thm comm sing}
and \ref{sin lp} which hold in spaces of homogeneous type, that%
\[
\left\Vert C_{a}f\right\Vert _{L^{p}\left(  F\right)  }\leq c\left\Vert
a\right\Vert _{BMO\left(  F\right)  }\left\Vert f\right\Vert _{L^{p}\left(
F\right)  }%
\]
for some constant $c$ depending on $p,n$ and the constants involved in the
assumptions on $K$ (but not on $R$); in turn, by Proposition
\ref{Prop BMO loc BMO} the last quantity is bounded by
\[
c\left\Vert a\right\Vert _{BMO_{loc}\left(  \Omega_{n+2},\Omega_{n+3}\right)
}\left\Vert f\right\Vert _{L^{p}\left(  F\right)  }.
\]
Reasoning on the support of $f$ we get, like in the proof of Theorem
\ref{Theorem L^p C^eta},%
\[
\left\Vert C_{a}f\right\Vert _{L^{p}\left(  B\left(  \overline{x},R\right)
\right)  }\leq c\left\Vert a\right\Vert _{BMO_{loc}\left(  \Omega_{n+2}%
,\Omega_{n+3}\right)  }\left\Vert f\right\Vert _{L^{p}\left(  B\left(
\overline{x},R\right)  \right)  }.
\]
To prove the second assertion, we now observe that if we apply the $L^{p}$
continuity estimate%
\begin{equation}
\left\Vert Tf\right\Vert _{L^{p}\left(  B\left(  \overline{x},R\right)
\right)  }\leq c\left\Vert f\right\Vert _{L^{p}\left(  B\left(  \overline
{x},R\right)  \right)  } \label{T on B_R}%
\end{equation}
to functions $f\in L^{p}\left(  B\left(  \overline{x},r\right)  \right)  $ for
any $r<R,$ we find%
\begin{equation}
\left\Vert Tf\right\Vert _{L^{p}\left(  B\left(  \overline{x},r\right)
\right)  }\leq c\left\Vert f\right\Vert _{L^{p}\left(  B\left(  \overline
{x},r\right)  \right)  } \label{T on B_r}%
\end{equation}
so that the same operator $T$ is continuous on $L^{p}\left(  B\left(
\overline{x},r\right)  \right)  ,$ for any $r<R,$ with a constant independent
of $r$. (Recall that the number $R$ is involved in the definition of $T$
(through the cutoff functions), so that (\ref{T on B_r}) is not the same as
\textquotedblleft(\ref{T on B_R}) for $R$ small\textquotedblright).

Take $r$ so small that for $x,y\in B\left(  \overline{x},r\right)  $ we have
$\widetilde{K}\left(  x,y\right)  =K\left(  x,y\right)  .$ Then the kernel
$\widetilde{K}$ satisfies in $B\left(  \overline{x},r\right)  $\ assumptions
(i) in Proposition \ref{Prop check assumptions}, with constants independent of
$r$. We can therefore apply again the commutator theorem on spaces of
homogeneous type (Theorem \ref{Thm comm sing}) to the operator $T$ on the
space $F^{\prime}$ built as in Theorem \ref{Thm F} with $F^{\prime}$
essentially containing $B\left(  \overline{x},r\right)  $ and and comparable
with it, concluding that, for any $f\in L^{p}\left(  B\left(  \overline
{x},r\right)  \right)  ,$%
\begin{align}
\left\Vert C_{a}f\right\Vert _{L^{p}\left(  B\left(  \overline{x},r\right)
\right)  }  &  \leq\left\Vert C_{a}f\right\Vert _{L^{p}\left(  F^{\prime
}\right)  }\leq c\left\Vert a\right\Vert _{BMO\left(  F^{\prime}\right)
}\left\Vert f\right\Vert _{L^{p}\left(  F^{\prime}\right)  }\nonumber\\
&  \leq c\eta_{a,\Omega_{n+2},\Omega_{n+3}}^{\ast}\left(  c_{n}r\right)
\left\Vert f\right\Vert _{L^{p}\left(  B\left(  \overline{x},r\right)
\right)  } \label{comm last}%
\end{align}
where we have applied again Proposition \ref{Prop BMO loc BMO}. Since in the
last inequality the constants $c,c_{n}$ are independent of $r$, if $a\in
VMO_{loc}\left(  \Omega_{n+2},\Omega_{n+3}\right)  ,$ for any $\varepsilon>0$
we can find $r$ small enough so that $c\eta_{a,\Omega_{n+2},\Omega_{n+3}%
}^{\ast}\left(  c_{1}r\right)  <\varepsilon,$ and we are done.
\end{proof}

\begin{theorem}
[Positive commutators of local fractional integrals]\label{Thm comm frac}Let
$K,\widetilde{K}$ be as in Assumption (H7), with $K$ satisfying the growth
condition (\ref{Standard 1'}) for some $\nu>0.$ If%
\[
I_{\nu}f\left(  x\right)  =\int_{B\left(  \overline{x},R\right)  }%
\widetilde{K}(x,y)f\left(  y\right)  d\mu(y)
\]
and, for $a\in BMO_{loc}\left(  \Omega_{n+2},\Omega_{n+3}\right)  $, we set%
\begin{equation}
C_{\nu,a}f\left(  x\right)  =\int_{B\left(  \overline{x},R\right)  }%
\widetilde{K}(x,y)\left\vert a\left(  x\right)  -a\left(  y\right)
\right\vert f\left(  y\right)  d\mu(y) \label{positive comm}%
\end{equation}
then, for any $p\in\left(  1,\frac{1}{\nu}\right)  ,\frac{1}{q}=\frac{1}%
{p}-\nu$ there exists $c\ $such that%
\[
\left\Vert C_{\nu,a}f\right\Vert _{L^{q}\left(  B\left(  \overline
{x},R\right)  \right)  }\leq c\left\Vert a\right\Vert _{BMO_{loc}\left(
\Omega_{n+2},\Omega_{n+3}\right)  }\left\Vert f\right\Vert _{L^{p}\left(
B\left(  \overline{x},R\right)  \right)  }%
\]
for any $f\in L^{p}\left(  B\left(  \overline{x},R\right)  \right)  .$

Moreover, if $a\in VMO_{loc}\left(  \Omega_{n+2},\Omega_{n+3}\right)  $ for
any $\varepsilon>0$ there exists $r>0$ such that for any $f\in L^{p}\left(
B\left(  \overline{x},r\right)  \right)  $ we have%
\[
\left\Vert C_{\nu,a}f\right\Vert _{L^{q}\left(  B\left(  \overline
{x},r\right)  \right)  }\leq\varepsilon\left\Vert f\right\Vert _{L^{p}\left(
B\left(  \overline{x},r\right)  \right)  }.
\]
The constant $c$ depends on $p,\nu,n$ and the constants involved in the
assumptions on $K$ (but not on $R$); the constant $r$ also depends on the
$VMO_{loc}\left(  \Omega_{n+2},\Omega_{n+3}\right)  $ modulus of $a$.
\end{theorem}

\begin{proof}
Proceeding like in the proof of Theorem \ref{Thm commutator}, and with the
same meaning of the symbols, applying Theorem \ref{Thm comm frac hom} which
holds in spaces of homogeneous type, we get that%
\begin{align*}
\left\Vert C_{\nu,a}f\right\Vert _{L^{q}\left(  F\right)  }  &  \leq
c\left\Vert a\right\Vert _{BMO\left(  F\right)  }\left\Vert f\right\Vert
_{L^{p}\left(  F\right)  }\\
&  \leq c\left\Vert a\right\Vert _{BMO_{loc}\left(  \Omega_{n+2},\Omega
_{n+3}\right)  }\left\Vert f\right\Vert _{L^{p}\left(  F\right)  }%
\end{align*}
for any $p\in\left(  1,\frac{1}{\nu}\right)  ,\frac{1}{q}=\frac{1}{p}-\nu$.
Like in the proof of Theorem \ref{Thm commutator}, we have
\[
\left\Vert C_{\nu,a}f\right\Vert _{L^{q}\left(  B\left(  \overline
{x},R\right)  \right)  }\leq c\left\Vert a\right\Vert _{BMO_{loc}\left(
\Omega_{n+2},\Omega_{n+3}\right)  }\left\Vert f\right\Vert _{L^{p}\left(
B\left(  \overline{x},R\right)  \right)  }.
\]
for any $f\in L^{p}\left(  B\left(  \overline{x},R\right)  \right)  $. The
same argument in the proof of Theorem \ref{Thm commutator} also gives%
\begin{align*}
\left\Vert C_{\nu,a}f\right\Vert _{L^{p}\left(  B\left(  \overline
{x},r\right)  \right)  }  &  \leq c\eta_{a,\Omega_{n+2},\Omega_{n+3}}^{\ast
}\left(  c_{1}r\right)  \left\Vert f\right\Vert _{L^{q}\left(  B\left(
\overline{x},r\right)  \right)  }\\
&  \leq\varepsilon\left\Vert f\right\Vert _{L^{q}\left(  B\left(  \overline
{x},r\right)  \right)  }%
\end{align*}
for $a\in VMO_{loc}\left(  \Omega_{n+2},\Omega_{n+3}\right)  ,$ and $r$ small
enough so that $c\eta_{a,\Omega_{n+2},\Omega_{n+3}}^{\ast}\left(
c_{1}r\right)  <\varepsilon,$ so we are done.
\end{proof}

\begin{theorem}
[Positive commutators of nonsingular integrals]\label{Thm comm pos}Let
$K,\widetilde{K}$ be as in Assumption (H7), with $K$ satisfying condition
(\ref{standard 2}) with $\nu=0.$ Assume that the operator%
\[
Tf\left(  x\right)  =\int_{B\left(  \overline{x},R\right)  }\widetilde
{K}(x,y)f\left(  y\right)  d\mu(y)
\]
is continuous on $L^{p}\left(  B\left(  \overline{x},R\right)  \right)  $ for
any $p\in\left(  1,\infty\right)  $. For $a\in BMO_{loc}\left(  \Omega
_{n+2},\Omega_{n+3}\right)  ,$ set%
\begin{equation}
C_{a}f\left(  x\right)  =\int_{B\left(  \overline{x},R\right)  }\widetilde
{K}(x,y)\left\vert a\left(  x\right)  -a\left(  y\right)  \right\vert f\left(
y\right)  d\mu(y), \label{positive comm 2}%
\end{equation}
then%
\[
\left\Vert C_{a}f\right\Vert _{L^{p}\left(  B\left(  \overline{x},R\right)
\right)  }\leq c\left\Vert a\right\Vert _{BMO_{loc}\left(  \Omega_{n+2}%
,\Omega_{n+3}\right)  }\left\Vert f\right\Vert _{L^{p}\left(  B\left(
\overline{x},R\right)  \right)  }%
\]
for any $f\in L^{p}\left(  B\left(  \overline{x},R\right)  \right)
,p\in\left(  1,\infty\right)  $.

Moreover, if $a\in VMO_{loc}\left(  \Omega_{n+2},\Omega_{n+3}\right)  $ for
any $\varepsilon>0$ there exists $r>0$ such that for any $f\in L^{p}\left(
B\left(  \overline{x},r\right)  \right)  $ we have%
\[
\left\Vert C_{a}f\right\Vert _{L^{p}\left(  B\left(  \overline{x},r\right)
\right)  }\leq\varepsilon\left\Vert f\right\Vert _{L^{p}\left(  B\left(
\overline{x},r\right)  \right)  }.
\]
The constant $c$ depends on $n,$ the constants involved in the assumptions on
$K,$ and the $L^{p}$-$L^{p}$ norm of the operator $T$ (but not explicitly on
$R$); the constant $r$ also depends on the $VMO_{loc}\left(  \Omega
_{n+2},\Omega_{n+3}\right)  $ modulus of $a$.
\end{theorem}

\begin{proof}
The proof is very similar to that of Theorem \ref{Thm comm frac}. Here we need
to apply Theorem \ref{Thm comm pos hom} which holds in spaces of homogeneous type.
\end{proof}

\begin{remark}
The presence of an absolute value inside the integral in (\ref{positive comm})
and (\ref{positive comm 2}) make the corresponding commutator theorems more
flexible than the analogue for singular integrals. Namely, if Theorem
\ref{Thm comm frac} or \ref{Thm comm pos} applies to a kernel $K,$ it also
applies to any other positive kernel equivalent to $K,$ differently from what
happens for singular integrals, for which the cancellation property is crucial.
\end{remark}

\section{Local maximal operators\label{Section Maximal}}

In this section we briefly deal with the local maximal operator in locally
homogeneous spaces. The result is substantially already known (see for
instance \cite{K}), but for the sake of completeness we state it explicitly
with the language and notation of this paper.

\begin{definition}
Fix $\Omega_{n},\Omega_{n+1}$ and, for any $f\in L^{1}\left(  \Omega
_{n+1}\right)  $ define the \emph{local maximal function}%
\[
M_{\Omega_{n},\Omega_{n+1}}f\left(  x\right)  =\sup_{r\leq r_{n}}\frac{1}%
{\mu\left(  B\left(  x,r\right)  \right)  }\int_{B\left(  x,r\right)
}\left\vert f\left(  y\right)  \right\vert d\mu\left(  y\right)  \text{ for
}x\in\Omega_{n}%
\]
where $r_{n}=2\varepsilon_{n}/\left(  2B_{n}+3B_{n}^{2}\right)  ,$ with
$B_{n}$ as in (\ref{Hp 2}).
\end{definition}

The following Vitali covering lemma holds, with the usual proof (see e.g.
\cite[Chap. 3]{CW}), thanks to the fact that by our restriction on $x$ and $r$
we can apply the local doubling condition to the involved balls:

\begin{lemma}
\label{Vitali cover lemma}Let $E$ be a measurable subset of $\Omega_{n}$ that
is covered by the union of a family of balls $B\left(  x_{\alpha},r_{\alpha
}\right)  $ centered at points of $\Omega_{n}$ and with radii $r_{\alpha}\leq
r_{n}.$ Then one can select a disjoint countable subcollection $\left\{
B\left(  x_{\alpha_{j}},r_{\alpha_{j}}\right)  \right\}  _{j=1}^{\infty}$ so
that
\[
E\subset%
{\displaystyle\bigcup\limits_{j=1}^{\infty}}
B\left(  x_{\alpha_{j}},Kr_{\alpha_{j}}\right)  \text{ with }K=\left(
2B_{n}+3B_{n}^{2}\right)  ,
\]
and, for some constant $c$ depending on $n,$%
\[
\underset{j=1}{\overset{\infty}{\sum}}\mu\left(  B\left(  x_{\alpha_{j}%
},r_{\alpha_{j}}\right)  \right)  \geq c\mu\left(  E\right)  .
\]

\end{lemma}

Then, again repeating the standard proof, one can establish the following:

\begin{theorem}
\label{Thm maximal}Let $f$ be a measurable function defined on $\Omega_{n+1}.$
The following hold:

(a) If $f\in L^{p}\left(  \Omega_{n+1}\right)  $ for some $p\in\left[
1,\infty\right]  $, then $M_{\Omega_{n},\Omega_{n+1}}f$ is finite almost
everywhere in $\Omega_{n}$;

(b) if $f\in L^{1}\left(  \Omega_{n+1}\right)  $, then for every $t>0$,
\[
\mu\left(  \left\{  x\in\Omega_{n}:\left(  M_{\Omega_{n},\Omega_{n+1}%
}f\right)  \left(  x\right)  >t\right\}  \right)  \leq\frac{c_{n}}{t}%
\int_{\Omega_{n+1}}\left\vert f\left(  y\right)  \right\vert d\mu\left(
y\right)  ;
\]

(c) if $f\in L^{p}\left(  \Omega_{n+1}\right)  $, $1<p\leq\infty$, then
$M_{\Omega_{n},\Omega_{n+1}}f\in L^{p}\left(  \Omega_{n}\right)  $ and
\[
\left\Vert M_{\Omega_{n},\Omega_{n+1}}f\right\Vert _{L^{p}\left(  \Omega
_{n}\right)  }\leq c_{n,p}\left\Vert f\right\Vert _{L^{p}\left(  \Omega
_{n+1}\right)  }.
\]

\end{theorem}

\section{Quasisymmetric quasidistances\label{section quasisymmetric}}

In this section we want to extend the main results of the previous theory to
the more general case of a quasisymmetric $\rho.$ We stress the fact that the
results we are going to extend are those of Sections \ref{section singular} to
\ref{Section Maximal}, but not the construction of dyadic cubes of Section
\ref{Section Dyadic}.

\begin{definition}
[Quasisymmetric locally homogeneous space]We make the following assumptions.

(K1) Let $\Omega$ be a set, endowed with a function $\rho:\Omega\times
\Omega\rightarrow\lbrack0,\infty)$ such that for any $x,y\in\Omega$
$\rho\left(  x,y\right)  =0\Leftrightarrow x=y.$

For any $x\in\Omega,r>0,$ let us define the ball%
\[
B\left(  x,r\right)  =\left\{  y\in\Omega:\rho\left(  x,y\right)  <r\right\}
\]
and the coball%
\[
B^{\prime}\left(  x,r\right)  =\left\{  y\in\Omega:\rho\left(  y,x\right)
<r\right\}  .
\]
Let us define a topology in $\Omega$ saying that $A\subset\Omega$ is open if
for any $x\in A$ there exists $r>0$ such that $B\left(  x,r\right)  \subset
A$. Also, we will say that $E\subset\Omega$ is bounded if $E$ is contained in
some ball. Let us assume that:

(K2') $\rho\left(  x,y\right)  $ is a continuous function of $x$ for any fixed
$y\in\Omega$ and a continuous function of $y$ for any fixed $x\in\Omega$.

(H3) Let $\mu$ be a positive regular Borel measure in $\Omega.$

(K4) Assume there exists an increasing sequence $\left\{  \Omega_{n}\right\}
_{n=1}^{\infty}$ of bounded measurable subsets of $\Omega,$ such that:%
\[%
{\displaystyle\bigcup_{n=1}^{\infty}}
\Omega_{n}=\Omega
\]
and such for, any $n=1,2,3,...$:

(i) the closure of $\Omega_{n}$ in $\Omega$ is compact;

(ii) there exists $\varepsilon_{n}>0$ such that%
\begin{align*}
\left\{  x\in\Omega:\rho\left(  x,y\right)  <2\varepsilon_{n}\text{ for some
}y\in\Omega_{n}\right\}   &  \subset\Omega_{n+1};\\
\left\{  x\in\Omega:\rho\left(  y,x\right)  <2\varepsilon_{n}\text{ for some
}y\in\Omega_{n}\right\}   &  \subset\Omega_{n+1};
\end{align*}
\qquad

(K5) there exist $A_{n},B_{n}\geq1$ such that for any $x,y,z\in\Omega_{n}$%
\begin{align*}
\rho\left(  x,y\right)   &  \leq A_{n}\rho\left(  y,x\right)  ;\\
\rho\left(  x,y\right)   &  \leq B_{n}\left(  \rho\left(  x,z\right)
+\rho\left(  z,y\right)  \right)  ;
\end{align*}

(H6) there exists $C_{n}>1$ such that for any $x\in\Omega_{n},0<r\leq
\varepsilon_{n}$ we have%
\[
\text{ }0<\mu\left(  B\left(  x,2r\right)  \right)  \leq C_{n}\mu\left(
B\left(  x,r\right)  \right)  <\infty.
\]
(Note that for $x\in\Omega_{n}$ and $r\leq\varepsilon_{n}$ we also have
$B\left(  x,2r\right)  \subset\Omega_{n+1}$).

We will say that $\left(  \Omega,\left\{  \Omega_{n}\right\}  _{n=1}^{\infty
},\rho,\mu\right)  $ is a \emph{quasisymmetric locally homogeneous space }if
assumptions (K1), (K2'), (H3), (K4), (K5), (H6) hold.
\end{definition}

With a proof very similar to that of Proposition \ref{Prop topology} in
Section \ref{section abstract} we can prove the following:

\begin{proposition}
Condition (K2') is equivalent to the validity of both the following

(K2) (a) the balls and coballs are open with respect to this topology;

(K2) (b) for any $x\in\Omega$ and $r>0$ the closure of $B\left(  x,r\right)  $
is contained in $\left\{  y\in\Omega:\rho\left(  x,y\right)  \leq r\right\}  $
and the closure of $B^{\prime}\left(  x,r\right)  $ is contained in $\left\{
y\in\Omega:\rho\left(  y,x\right)  \leq r\right\}  .$
\end{proposition}

It is also immediate to check the following

\begin{proposition}
If $\left(  \Omega,\left\{  \Omega_{n}\right\}  _{n=1}^{\infty},\rho
,\mu\right)  $ is a quasisymmetric locally homogeneous space and%
\[
\rho^{\ast}\left(  x,y\right)  =\rho\left(  x,y\right)  +\rho\left(
y,x\right)  ,
\]
then $\left(  \Omega,\left\{  \Omega_{n}\right\}  _{n=1}^{\infty},\rho^{\ast
},\mu\right)  $ is a locally homogeneous space, and its constants can be
bounded in terms of the constants of $\left(  \Omega,\left\{  \Omega
_{n}\right\}  _{n=1}^{\infty},\rho,\mu\right)  .$
\end{proposition}

We now want to apply the results we have proved in Sections
\ref{section singular} to \ref{Section Maximal} to show that similar results
hold in a quasisymmetric locally homogeneous space. Let us discuss in detail
one of these results, the others being similar.

\begin{theorem}
[$L^{p}$ and $C^{\eta}$ estimates for singular integrals]Theorem
\ref{Theorem L^p C^eta} still holds if $\left(  \Omega,\left\{  \Omega
_{n}\right\}  _{n=1}^{\infty},\rho,\mu\right)  $ is a quasisymmetric locally
homogeneous space.
\end{theorem}

\begin{proof}
The key observation is that if the kernel $K$ satisfies conditions (i), (ii),
(iii) in Proposition \ref{Prop check assumptions} with respect to $\rho,$ it
also satisfies them with respect to any equivalent function, in particular
with respect to $\rho^{\ast};$ this follows by a standard computation, and
implies the validity of $L^{p}$ estimates, by Theorem \ref{Theorem L^p C^eta}.
As to $C^{\eta}$ estimates, let us first note that $\rho$ and $\rho^{\ast}$
define the same space $C^{\eta},$ with equivalent norms. Moreover, if $K$
satisfies the condition
\[
\widetilde{h}\left(  x\right)  \equiv\lim_{\varepsilon\rightarrow0}\int
_{\rho^{\prime}(x,y)>\varepsilon}\widetilde{K}(x,y)d\mu(y)\in C^{\gamma
}\left(  \Omega_{n+1}\right)
\]
for some $\gamma>0$ and some $\rho^{\prime}$ equivalent to $\rho,$ this
$\rho^{\prime}$ is also equivalent to $\rho^{\ast},$ so $\widetilde{h}$
satisfies the H\"{o}lder continuity assumption required by Theorem
\ref{Theorem L^p C^eta}, hence H\"{o}lder estimates hold.
\end{proof}

A similar argument shows that Theorems \ref{frac lp-lq} and
\ref{Thm frac C^eta} still holds if $\left(  \Omega,\left\{  \Omega
_{n}\right\}  _{n=1}^{\infty},\rho,\mu\right)  $ is a quasisymmetric locally
homogeneous space.

To deal with commutators we first need to make the following remark about
$BMO$ spaces.

Let us denote by $B_{r},\widetilde{B}_{r}$ the balls with respect to any two
equivalent functions $\rho,\widetilde{\rho}$ satisfying the axioms of
quasisymmetric locally doubling spaces. Then for any $x_{0}\in\Omega_{n}%
,r\leq\varepsilon_{n},$ any $\tau\in\mathbb{R},$ we can write, by the
equivalence of $\rho,\widetilde{\rho}$%
\[
\frac{1}{\left\vert B_{r}\left(  x_{0}\right)  \right\vert }\int_{B_{r}\left(
x_{0}\right)  }\left\vert u\left(  x\right)  -\tau\right\vert d\mu\left(
x\right)  \leq\frac{1}{\left\vert \widetilde{B}_{c_{1}r}\left(  x_{0}\right)
\right\vert }\int_{\widetilde{B}_{c_{2}r}\left(  x_{0}\right)  }\left\vert
u\left(  x\right)  -\tau\right\vert d\mu\left(  x\right)
\]
by the local doubling condition%
\[
\leq\frac{c_{3}}{\left\vert \widetilde{B}_{c_{2}r}\left(  x_{0}\right)
\right\vert }\int_{\widetilde{B}_{c_{2}r}\left(  x_{0}\right)  }\left\vert
u\left(  x\right)  -\tau\right\vert d\mu\left(  x\right)  .
\]
Choosing $\tau=\frac{1}{\left\vert \widetilde{B}_{c_{2}r}\left(  x_{0}\right)
\right\vert }\int_{\widetilde{B}_{c_{2}r}\left(  x_{0}\right)  }u\left(
x\right)  d\mu\left(  x\right)  $ and recalling that for any $\tau$ we have%
\[
\frac{1}{\left\vert B_{r}\left(  x_{0}\right)  \right\vert }\int_{B_{r}\left(
x_{0}\right)  }\left\vert u\left(  x\right)  -u_{B_{r}\left(  x_{0}\right)
}\right\vert d\mu\left(  x\right)  \leq2\cdot\frac{1}{\left\vert B_{r}\left(
x_{0}\right)  \right\vert }\int_{B_{r}\left(  x_{0}\right)  }\left\vert
u\left(  x\right)  -\tau\right\vert d\mu\left(  x\right)
\]
we get the equivalence between the norms $\left\Vert u\right\Vert
_{BMO_{loc}\left(  \Omega_{n},\Omega_{n+1}\right)  }$ with respect to $\rho$
and $\widetilde{\rho},$ and an analogous equivalence between $VMO_{loc}$
moduli. Applying this argument to the quasisymmetric function $\rho$ and its
symmetrized $\rho^{\ast}$ we immediately get that also the commutator theorems
\ref{Thm commutator}, \ref{Thm comm frac}, \ref{Thm comm pos} still hold if
$\left(  \Omega,\left\{  \Omega_{n}\right\}  _{n=1}^{\infty},\rho,\mu\right)
$ is a quasisymmetric locally homogeneous space.

Finally, the extension of Theorem \ref{Thm maximal} to the setting of
quasisymmetric locally homogeneous spaces is immediate, since the maximal
functions defined with respect to equivalent quasisymmetric quasidistances are
equivalent, hence the result in the symmetric case implies that for the
quasisymmetric case.

\section{Appendix. Known results for spaces of homogeneous
type\label{section appendix}}

In this Appendix we collect all the results about spaces of homogeneous type
which we have applied throughout the paper. We first recall the basic

\begin{definition}
Let $X$ be a set endowed with a function $\rho:X\times X\rightarrow
\lbrack0,\infty)$ such that for some constant $B\geq1$, any $x,y,z\in X$:

$\rho\left(  x,y\right)  =0\Longleftrightarrow x=y;$

$\rho\left(  x,y\right)  =\rho\left(  y,x\right)  ;$

$\rho\left(  x,y\right)  \leq B\left(  \rho\left(  x,z\right)  +\rho\left(
z,y\right)  \right)  .$

Assume that the $\rho$-balls are open with respect to the topology they
induce. Let $\mu$ be a positive Borel measure on $X$, satisfying the doubling
condition%
\[
0<\mu\left(  B\left(  x,2r\right)  \right)  \leq C\mu\left(  B\left(
x,r\right)  \right)  <\infty
\]
for any $x\in X,r>0.$ Then we say that $\left(  X,\rho,\mu\right)  $ is a
space of homogeneous type.
\end{definition}

\textbf{Dependence of the constants. }We will say that some constant depends
on $X$ to say that it depends on the constants $B,C.$

\subsection{$L^{p}$ and $C^{\alpha}$ estimates for singular integrals on
spaces of homogeneous type}

\begin{theorem}
[$L^{p}$ continuity of singular integrals]\label{sin lp}Let $\left(
X,\rho,\mu\right)  $ be a homogeneous space, $\mu$ a regular measure. Let
$K:X\times X\backslash\left\{  x=y\right\}  \rightarrow\mathbb{R}$ a kernel
satisfying the following conditions:

the standard estimates (\ref{standard 1}) with $\nu=0$, for any $x,y\in X$,
and (\ref{standard 2}), for any $x_{0},x,y\in X,$ with $\rho\left(
x_{0},y\right)  \geq M\rho\left(  x_{0},x\right)  ,M>1,\nu=0,\beta>0;$

the cancellation property%
\[
\left\vert \int_{r\leq\rho^{\prime}\left(  x,y\right)  \leq R}K\left(
y,x\right)  d\mu\left(  y\right)  \right\vert +\left\vert \int_{r\leq
\rho^{\prime}\left(  x,y\right)  \leq R}K\left(  x,y\right)  d\mu\left(
y\right)  \right\vert \leq C
\]
for any $R>r>0,x\in X,$ where $\rho^{\prime}$ is any quasidistance equivalent
to $\rho$. Let $T_{\varepsilon}$ be the truncated operator defined as%
\[
T_{\varepsilon}f=\int_{\rho^{\prime}\left(  x,y\right)  \geq\varepsilon
}K\left(  x,y\right)  f\left(  y\right)  d\mu\left(  y\right)
\]
for any $f\in C_{0}^{\eta}\left(  X\right)  $ (with $\eta$ small enough so
that $C_{0}^{\eta}\left(  X\right)  $ is dense in $L^{p}$ for $p\in
\lbrack1,\infty)$). Then $T_{\varepsilon}$ can be extended to a linear
continuous operator from $L^{p}$\ into $L^{p}$ for every $p\in\left(
1,\infty\right)  $, and%
\[
\left\Vert T_{\varepsilon}f\right\Vert _{L^{p}}\leq c\left\Vert f\right\Vert
_{L^{p}}%
\]
where the constant $c$ depends on $X$, $p$ and all the constants involved in
the assumptions, but not on $\varepsilon$. Moreover, if for a.e. $x\in X$
there exists the limit%
\[
\lim_{\varepsilon\rightarrow0}\int_{\rho^{\prime}\left(  x,y\right)
\geq\varepsilon}K\left(  x,y\right)  d\mu\left(  y\right)  ,
\]
then the above $L^{p}$ estimate holds also for the operator%
\[
Tf\left(  x\right)  =\lim_{\varepsilon\rightarrow0}\int_{\rho^{\prime}\left(
x,y\right)  \geq\varepsilon}K\left(  x,y\right)  f\left(  y\right)
d\mu\left(  y\right)  .
\]
Finally, the operator $T$ satisfies a weak $\left(  1,1\right)  $-estimate:%
\[
\mu\left(  \left\{  x\in X:\left\vert Tf\left(  x\right)  \right\vert
>t\right\}  \right)  \leq\frac{c}{t}\left\Vert f\right\Vert _{L^{1}\left(
X\right)  }.
\]

\end{theorem}

The above result follows, for instance, from the results contained in \cite{C}
and \cite{CW}; see also \cite[Thm. 4.1, Thm. 4.5]{BC1} where this theorem is
explicitly derived from the aforementioned results.

\begin{theorem}
[$C^{\alpha}$ continuity of singular integrals]\label{sin holder}(See
\cite[Thm. 2.7]{BB4}). Let $\left(  X,\rho,\mu\right)  $ be a bounded
homogeneous space, and let \thinspace$K$ $:X\times X\backslash\left\{
x=y\right\}  \rightarrow%
\mathbb{R}
$ a kernel satisfying the following conditions:

the standard estimate (\ref{standard 1}) with $\nu=0,$ any $x,y\in X,$ and
(\ref{standard 2}) with $\nu=0$, for any $x_{0},x,y\in X,$ with $\rho\left(
x_{0},y\right)  \geq M\rho\left(  x_{0},x\right)  ,M>1,\beta>0;$

the cancellation property: for any $r>0$%
\[
\left\vert \int_{\rho^{\prime}\left(  x,y\right)  \geq r}K\left(  y,x\right)
d\mu\left(  y\right)  \right\vert +\left\vert \int_{\rho^{\prime}\left(
x,y\right)  \geq r}K\left(  x,y\right)  d\mu\left(  y\right)  \right\vert \leq
C,
\]
where $\rho^{\prime}$ is any quasidistance on $X$, equivalent to $\rho$.
Assume that for every $f\in C^{\alpha}(X)$ and $x\in X$ the following limit
exists:%
\[
Tf\left(  x\right)  =\lim_{\varepsilon\rightarrow0}T_{\varepsilon}f\left(
x\right)  =\lim_{\varepsilon\rightarrow0}\int_{\rho^{\prime}\left(
x,y\right)  \geq\varepsilon}K\left(  x,y\right)  f\left(  y\right)  dy
\]
and $T\left(  1\right)  \left(  x\right)  \in C^{\eta}(X)$, for some $\eta
\in(0,1]$ Then the operator $T$ is continuous on $C^{\alpha}(X)$; more
precisely:%
\[
\left\vert Tf\right\vert _{C^{\alpha}(X)}\leq c\left\Vert f\right\Vert
_{C^{\alpha}(X)}\text{ for every }\alpha<\beta,\alpha\leq\eta
\]
for some constant $c$ depending on $X,\alpha,$ and all the constants involved
in the assumptions (recall $\beta$ is the exponent appearing in assumption
(\ref{standard 2})). Moreover,%
\[
\left\Vert Tf\right\Vert _{L^{\infty}}\leq c\left\Vert f\right\Vert
_{C^{\alpha}(X)},
\]
where $c$ also depending on diam$X$.
\end{theorem}

\subsection{$L^{p}$ and $C^{\alpha}$ estimates for fractional integrals on
spaces of homogeneous type}

\begin{theorem}
[$L^{p}$ estimate for fractional integrals]\label{frac lp}Let $\left(
X,\rho,\mu\right)  $ be a space of homogeneous type, and for $\alpha\in\left(
0,1\right)  ,$ let%
\[
0\leq K_{\alpha}\left(  x,y\right)  \leq\frac{C}{\mu\left(  B\left(
x,\rho\left(  x,y\right)  \right)  \right)  ^{1-\alpha}}\text{ for }x\neq
y,K_{\alpha}\left(  x,x\right)  =0
\]%
\[
I_{\alpha}f\left(  x\right)  =\int_{X}K_{\alpha}\left(  x,y\right)  f\left(
y\right)  d\mu\left(  y\right)
\]
for any measurable $f:X\rightarrow\mathbb{R}$ for which the integral makes
sense. Then, for any $p\in\left(  1,\frac{1}{\alpha}\right)  ,\frac{1}%
{q}=\frac{1}{p}-\alpha$ there exists a constant depending on $X,\alpha,p$ and
the constant $C$ in the assumptions on $K_{\alpha},$ such that%
\[
\left\Vert I_{\alpha}f\right\Vert _{L^{q}\left(  X\right)  }\leq c\left\Vert
f\right\Vert _{L^{p}\left(  X\right)  }%
\]
for any $f\in L^{p}\left(  X\right)  .$
\end{theorem}

The above result is due to Gatto-Vagi, see \cite{GV1}, \cite{GV2}.

\begin{theorem}
[$C^{\alpha}$ estimate for fractional integrals]\label{Thm frac C^alfa}(See
\cite[Thm 2.11]{BB4}). Let $\left(  X,\rho,\mu\right)  $ be a bounded space of
homogeneous type, and let $K\left(  x,y\right)  $ be a kernel satisfying for
some $\nu\in\left(  0,1\right)  $ the standard estimates (\ref{standard 1})
for any $x,y\in X$, and (\ref{standard 2}), for any $x_{0},x,y\in X,$ with
$\rho\left(  x_{0},y\right)  \geq M\rho\left(  x_{0},x\right)  ,M>1,\beta>0;$
let
\[
I_{\nu}f\left(  x\right)  =\int_{X}K_{\nu}\left(  x,y\right)  f\left(
y\right)  d\mu\left(  y\right)  .
\]
Then, for any $\alpha<\min\left(  \beta\,,\nu\right)  $ we have%
\[
\left\Vert I_{\nu}f\right\Vert _{C^{\alpha}\left(  X\right)  }\leq c\left\Vert
f\right\Vert _{C^{\alpha}\left(  X\right)  }.
\]
The constant $c$ depends on $X,\alpha,$diam$X,$ and the constants involved in
the assumptions on $K.$
\end{theorem}

The above two results about fractional integrals are proved in the quoted
papers under some additional assumptions on the space (e.g., the absence of
\textquotedblleft atoms\textquotedblright, that is points of positive
measure); however, these statements can be easily proved in full generality.

\subsection{Commutator theorems}

The original commutator theorem we are interested in is the one proved by
Coifman-Rochberg-Weiss \cite{CRW} for classical Calder\'{o}n-Zygmund
operators. The extension of this result to spaces of homogeneous type, both
bounded and unbounded, has been proved in \cite[Thm 2.5, Thm. 3.1]{BC1}:

\begin{theorem}
[Commutators of singular integrals]\label{Thm comm sing}Let $\left(
X,\rho,\mu\right)  $ be a homogeneous space and let all the assumptions of
Theorem \ref{sin lp} be in force. Let
\[
Tf\left(  x\right)  =\lim_{\varepsilon\rightarrow0}\int_{X,\rho^{\prime
}(x,y)>\varepsilon}K(x,y)f\left(  y\right)  d\mu(y)
\]
and, for $a\in BMO\left(  X\right)  $ let%
\[
C_{a}f\left(  x\right)  =T\left(  af\right)  \left(  x\right)  -a\left(
x\right)  Tf\left(  x\right)  .
\]
Then, for any $p\in\left(  1,\infty\right)  $%
\[
\left\Vert C_{a}f\right\Vert _{L^{p}\left(  X\right)  }\leq c\left\Vert
a\right\Vert _{BMO\left(  X\right)  }\left\Vert f\right\Vert _{L^{p}\left(
X\right)  }%
\]
for some constant $c$ depending on $X,p,$ and the constants involved in the
assumptions on $K,$ but not on $f,a.$
\end{theorem}

Next, let us recall the analog results for fractional or more general type of
nonsingular integral operators. A key point in the following result is the
presence of an absolute value inside the integral, which is allowed by the
positivity of the kernel:

\begin{theorem}
[Commutators of fractional integrals]\label{Thm comm frac hom}Let $\left(
X,\rho,\mu\right)  $ be a homogeneous space and let $K_{\alpha},I_{\alpha}$ be
as in Theorem \ref{frac lp}. For any function $a\in BMO\left(  X\right)  ,$
let
\[
C_{a}f\left(  x\right)  =\int_{X}K_{\alpha}\left(  x,y\right)  \left\vert
a\left(  x\right)  -a\left(  y\right)  \right\vert f\left(  y\right)
d\mu\left(  y\right)
\]
be the \textquotedblleft positive commutator\textquotedblright\ of $I_{\alpha
}$ with $a$. Then for any $p\in\left(  1,\frac{1}{\alpha}\right)  ,\frac{1}%
{q}=\frac{1}{p}-\alpha$ there exists $c=c\left(  X,\alpha,p\right)  $ such
that%
\[
\left\Vert C_{a}f\right\Vert _{L^{q}\left(  X\right)  }\leq c\left\Vert
a\right\Vert _{BMO\left(  X\right)  }\left\Vert f\right\Vert _{L^{p}\left(
X\right)  }.
\]

\end{theorem}

The above theorem has been first proved in \cite[Thm. 2.11]{BC2} under an
extra assumption on the space $\left(  X,\rho\right)  $ and then, in this full
generality, in \cite[Thm. 3.3, Thm. 3.7]{B1}.

An analog result holds for any abstract operator with positive kernel, which
we already know to be $L^{p}$ continuous:

\begin{theorem}
\label{Thm comm pos hom}(See \cite[Thm. 0.1]{B1}). Let $\left(  X,\rho
,\mu\right)  $ be a homogeneous space and $K\left(  x,y\right)  $ be a
nonnegative kernel such that the operator%
\[
Tf\left(  x\right)  =\int_{X\setminus\left\{  x\right\}  }K\left(  x,y\right)
f\left(  y\right)  d\mu\left(  y\right)
\]
maps $L^{p}\left(  X\right)  $ into $L^{p}\left(  X\right)  $ for $p\in\left(
1,\infty\right)  .$ Also, assume $K$ satisfies the standard estimate
(\ref{standard 2}) for $\nu=0,$ some $\beta>0,M>1$, any $x_{0},x,y\in X$ with
$\rho(x_{0},y)>M\rho(x_{0},x).$ For $a\in BMO\left(  X\right)  $, let
\[
C_{a}f\left(  x\right)  =\int_{X}K\left(  x,y\right)  \left\vert a\left(
x\right)  -a\left(  y\right)  \right\vert f\left(  y\right)  d\mu\left(
y\right)  .
\]
Then there exists $c=c\left(  X,p\right)  $ such that%
\[
\left\Vert C_{a}f\right\Vert _{L^{p}\left(  X\right)  }\leq c\left\Vert
a\right\Vert _{BMO\left(  X\right)  }\left\Vert f\right\Vert _{L^{p}\left(
X\right)  }.
\]

\end{theorem}

\bigskip

Marco Bramanti

Dipartimento di Matematica, Politecnico di Milano

Via Bonardi 9. 20133 Milano. ITALY

marco.bramanti@polimi.it

\bigskip

Maochun Zhu

Department of Applied Mathematics, Northwestern Polytechincal University

127 West Youyi Road. 710072 Xi'an. P. R. CHINA

zhumaochun2006@126.com

\end{document}